\documentclass[11pt,reqno]{amsart}
\usepackage{amssymb,mathrsfs,color,mathtools,empheq, verbatim, epstopdf}
\usepackage{pinlabel}
\mathtoolsset{showonlyrefs}
\usepackage{hyperref} 

\usepackage{enumerate, bbm}
\usepackage{tensor}

\usepackage{graphicx}
\usepackage{xcolor} 
\usepackage{tensor}
\usepackage{slashed}
\usepackage{cite}

\usepackage[shortlabels]{enumitem}

\usepackage{geometry}

\numberwithin{equation}{section}

\newtheorem{mainthm}{Theorem}
\newtheorem{thm}{Theorem}[section]

\newtheorem{lem}[thm]{Lemma}
\newtheorem{prop}[thm]{Proposition}

\theoremstyle{definition} 
\newtheorem{rem}[thm]{Remark}
\newtheorem{defn}[thm]{Definition}

\theoremstyle{remark}

\def\bN {\mathbb{N}}

\def\bS {\mathbb{S}}
\def\bT {\mathbb{T}}

\def\cE {\mathcal{E}}

\def\cY {\mathcal{Y}}

\newcommand{\la}{\langle}
\newcommand{\ra}{\rangle}
\newcommand{\La}{\big\langle}
\newcommand{\Ra}{\big\rangle}

\newcommand{\spn}{\operatorname{span}}

\renewcommand{\ker}{\operatorname{ker}}

\newcommand{\vd}{\mathrm{d}}
\newcommand{\udr}{\,r\vd r}

\definecolor{deepgreen}{cmyk}{1,0,1,0.5}

\newcommand{\E}{\mathcal{E}}

\newcommand{\cS}{\mathcal{S}}

\newcommand{\N}{\mathbb{N}}
\newcommand{\R}{\mathbb{R}}

\newcommand{\Z}{\mathbb{Z}}

\newcommand{\al}{\alpha}
\newcommand{\be}{\beta}

\newcommand{\lam}{\lambda}
\newcommand{\te}{\theta}

\newcommand{\ka}{\kappa}

\newcommand{\De}{\Delta}

\newcommand{\Lam}{\Lambda}

\newcommand{\p}{\partial}

\makeatletter

\newcommand{\Rmnum}[1]{\expandafter\@slowromancap\romannumeral #1@}
\makeatother

\newcommand{\ti}{\widetilde}

\newcommand{\abs}[1]{\left\lvert{#1}\right\rvert}

\newcommand{\Del}[1]{}

\newcommand{\mand}{{\ \ \text{and} \ \  }}

\newcommand{\mif}{{\ \ \text{if} \ \ }}
\newcommand{\mfor}{{\ \ \text{for} \ \ }}
\newcommand{\mas}{{\ \ \text{as} \ \ }}

\renewcommand\Re{\mathrm{Re}\,}

\definecolor{green}{rgb}{0,0.8,0} 

\renewcommand{\Re}{\mathrm{Re}}

\newcommand{\ud}{\mathrm{d}}

\newcommand{\alp}{\alpha}

\newcommand{\veps}{\varepsilon}

\newcommand{\bfd}{{\bf d}}

\newcommand{\bfi}{{\bf i}}

\newcommand{\bfq}{{\bf q}}

\newcommand{\bbN}{\mathbb N}

\newcommand{\bbR}{\mathbb R}
\newcommand{\bbS}{\mathbb S}
\newcommand{\bbT}{\mathbb T}

\newcommand{\calC}{\mathcal C}

\newcommand{\calE}{\mathcal E}

\newcommand{\calL}{\mathcal L}

\newcommand{\calR}{\mathcal R}

\newcommand{\calT}{\mathcal T}

\newcommand{\calW}{\mathcal W}

\newcommand{\calY}{\mathcal Y}
\newcommand{\calZ}{\mathcal Z}

\allowdisplaybreaks[4]

\begin{document}
	\parindent=0pt
	\title[Soliton resolution for energy-critical GL in the radial case]{Soliton resolution for the energy-critical nonlinear Ginzburg-Landau equation in the radial case}
	\author{Yuchen Yin}
    \address{School of Mathematical Sciences\\ University of Science and Technology of China\\ Hefei 230026\\ Anhui\\ China}
\email{yuchenyin@mail.ustc.edu.cn}
\begin{abstract}
We study the the energy critical non-linear Ginzburg-Landau equation $\partial_{t} u  =z\Delta u+z|u|^{\frac{4}{D-2}} u$ with $\Re z >0$ in dimension $D\geq 3$. We prove that every radial solution with finite energy norm resolves into a finite superposition of asymptotically decoupled copies of the ground state and free radiation continuously in time. 
\end{abstract}
\maketitle
\section{Introduction}\label{sec:1}
\subsection{Setting of the Problem} Let $D\geq 3$ be the dimension. We study the energy-critical semi-linear Ginzburg-Landau flow in $\mathbb{R}^{D}$,
\begin{align}\label{eqn:GL-non-radial}
\partial_{t} u & =z\Delta u+zf(u), \qquad f(u)=|u|^{\frac{4}{D-2}} u, 
\end{align}

where $\Re z>0$.  We also assume that $|z|=1$ by rescaling. Under radial symmetry, the equation ~\eqref{eqn:GL-non-radial} can be written in the following form

\begin{align}\label{eqn:GL}
\partial_{t} u & =z\p_r^2 u+z\frac{D-1}{r}\p_r u+zf(u), \qquad f(u)=|u|^{\frac{4}{D-2}} u, 
\end{align}
where $u=u(t, r) \in \mathbb{C}$,   $r=|x| \in(0, \infty)$ is the radial coordinate in $\mathbb{R}^{D}$, and $\Delta:=\partial_{r}^{2}+(D-1) r^{-1} \partial_{r}$ is the radial Laplacian in $\mathbb{R}^{D}$. 

In this paper we study solutions that remain uniformly bounded in the energy space $\mathcal{E}$, which is defined via the norm,
\begin{align*}
\|u\|_{\mathcal{E}}^{2}:=\int_{0}^{\infty}\left[\left(\partial_{r} u(r)\right)^{2}+\frac{(u(r))^{2}}{r^{2}}\right] r^{D-1} \mathrm{~d} r .
\end{align*}

Given finite energy data, we denote the maximal forward time of existence by \( T_+ > 0 \). The nonlinear energy functional associated with $\eqref{eqn:GL}$ is given by
\[
E(u) := \frac{1}{2}\int_0^\infty  \left| \partial_r u(r) \right|^2 r^{D-1} \mathrm{d}r - \frac{1}{\frac{2D}{D-2}} \int_0^\infty |u(r)|^{\frac{2D}{D-2}} r^{D-1} \mathrm{d}r.
\]

Solutions to $\eqref{eqn:GL}$ are invariant under the scaling
\[
u(t, r) \mapsto u_\lambda(t, r) := \frac{1}{\lambda^{\frac{D-2}{2}}} u \left( \frac{t}{\lambda^2}, \frac{r}{\lambda} \right), \quad \lambda > 0
\]
and $\eqref{eqn:GL}$ is called energy-critical since \( E(u) = E\left( u_\lambda \right) \). We have the following formal energy identity
\[
E(u(t_2)) + \Re z \int_{t_1}^{t_2} \left\| \mathcal{T}(u(t)) \right\|_{L^2}^2 \mathrm{d}t = E(u(t_1))
\]

for each $t_1,t_2\in [0,T_+)$, where \( \mathcal{T}(u) := \partial_r^2 u + \frac{D - 1}{r}\partial_r u + |u|^{\frac{4}{D - 2}}u \). 

The stationary solutions to ~\eqref{eqn:GL} are characterized by the stationary equation
\begin{align}
    \label{eqn:stationary}
    -\Delta W(x) = |W(x)|^{\frac{4}{D - 2}}W(x),\quad x \in \mathbb{R}^D.
\end{align}
It's well-known that the Aubin–Talenti solution \( W: \mathbb{R}^D \to \mathbb{R} \) in the form
\[
W(x) := \left( 1 + \frac{|x|^2}{D(D - 2)} \right)^{-\frac{D - 2}{2}}
\]
is the unique radial solution to ~\eqref{eqn:stationary} up to sign and scaling. We will abuse notation slightly and write $W(x)=W(r)$ with $r=|x|$.

\subsection{Main Results}

We prove the following theorem. 

\begin{mainthm}[Soliton resolution]
    \label{thm:main}
    
     Let $D \geq 3$ , and let ${u}(t)$ be the solution to $\eqref{eqn:GL}$ with initial data
     $u(0)=u_0\in \mathcal{E}$, defined on its maximal interval of existence $[0, T_{+})$.  

     Suppose that 
$$
\limsup _{t \rightarrow T_{+}}\|{u}(t)\|_{\mathcal{E}}<\infty ,
$$

then

 \emph{({Global solution})} If $T_{+}=\infty$, there exist a time $T_0>0$, an integer $N \geq 0$, continuous functions $\lambda_1(t), \ldots, \lambda_N(t) \in C^0\left([T_0, \infty),(0,\infty)\right)$, phases $\theta_1(t), \ldots, \theta_N(t) \in C^0([T_0,\infty),\bbS)$and  ${g}(t) \in \mathcal{E}$ defined by
$$
{u}(t)=\sum_{j=1}^N e^{i\theta_j(t)} {W}_{\lambda_j(t)}+{g}(t),
$$
such that
$$
\|{g}(t)\|_{\mathcal{E}}+\sum_{j=1}^N \frac{\lambda_j(t)}{\lambda_{j+1}(t)} \rightarrow 0 \text { as } t \rightarrow \infty,
$$
where above we use the convention that $\lambda_{N+1}(t)=\sqrt{t}$;

 \emph{({Blow-up solution})} If $T_{+}<\infty$, there exist a time $T_0<T_{+}$, an integer $N \geq 0$,  a body map  ${u}^* \in \mathcal{E}$,  continuous functions $\lambda_1(t), \ldots, \lambda_N(t) \in C^0\left([T_0, T_{+}),(0,\infty)\right)$, phases $\theta_1(t), \ldots, \theta_N(t) \in C^0([T_0,T_{+}),\bbS)$ and ${g}(t) \in \mathcal{E}$ defined by
$$
{u}(t)=\sum_{j=1}^N e^{i\theta_j(t)}{W}_{\lambda_j(t)}+{u}^*+{g}(t),
$$
such that
$$
\|{g}(t)\|_{\mathcal{E}}+\sum_{j=1}^N \frac{\lambda_j(t)}{\lambda_{j+1}(t)} \rightarrow 0 \text { as } t \rightarrow T_{+},
$$
where above we use the convention that $\lambda_{N+1}(t)=\sqrt{T_{+}-t}$.   
\end{mainthm}

\begin{rem}
    The Soliton Resolution Conjecture predicts that generic solutions to nonlinear dispersive equations will decouple into a superposition of modulated solitons , a radiative term and a term going to zero in energy space asymptotically in time. This conjecture arose from numerical simulations, e.g., \cite{FPU55}, \cite{ZK65}, the theory of integrable systems,  \cite{Sch86}, and the bubbling theory of harmonic maps in the elliptic and parabolic settings \cite{Stru85}, \cite{Qin95}, \cite{QT97}, \cite{To04}. 

The Soliton Resolution Conjecture has been established for some classes of PDEs. For the focusing energy critical nonlinear wave equations, the full resolution was established for radial data in \cite{DKM23}, \cite{JL23a}, while the conjecture for general data was only established locally in space in $D=3,4,5$ for a well-chosen sequence of time in \cite{DJKM17}. For energy critical wave maps from $\R^2$ to $\bS^2$, the full resolution was established for equivariant solutions in \cite{JL25}, while the conjecture without the symmetry assumption was only established continuously in time for small blow up solutions in \cite{DJKM18}. In the parabolic setting, the conjecture was proven for the equivariant harmonic map heat flow \cite{JL23b} and for the radial focusing energy-critical heat equation \cite{Ary25ra}. The conjecture holds up to subsequence for any time sequences for harmonic map heat flow and focusing energy critical heat equation, see \cite{JLS25}, \cite{Ary25}. For results on other types of equations, see  \cite{KKO23}, \cite{KK24}.
\end{rem}

\begin{rem}
    The general Ginzburg-Landau equation is of the following form:
    \begin{equation}
        \p_t u=\gamma u+(a+i\al)\De u-(b+i\be)|u|^{p-1}u,
    \end{equation}

    where $u(t,x)$ is a complex valued function of $(t,x)\in \R_+\times \R^D$, $a,b,\al,\be\in \R$, and $p>1$. This class of equations has drawn the attention of many mathematicians and physicists due to its close relation to many phenomena in fluid dynamics and superconductivity, etc. 

    In this article, we focus on the special case ~\eqref{eqn:GL}, as the study concerning the existence and uniqueness of ground states for the general Ginzburg-Landau equation is complicated. Indeed, such a kind of special Ginzburg-Landau equation has been studied by Cazenave \cite{Caz13}, see also \cite{CCZ23}.

\end{rem}

\begin{rem}
    The proof of Theorem \ref{thm:main} closely follows that in \cite{Ary25ra}. This similarity stems from the analogous behavior of the Ginzburg-Landau and heat equations in the energy space. However, this analogy does not directly carry over to the geometric flow setting: obtaining $L^{\infty}$ estimates for Landau-Lifshitz flows seems to be more involved than for heat flows. We plan to address this in future work.
\end{rem}

\subsection{Strategy of the Proof}

We first prove the resolution along a well-chosen sequence of time, which requires a compactness lemma and characterization of the body map. Then we argue by contradiction to prove the resolution in continuous time. The main tools we use are modulation analysis and analysis for collision intervals developed in \cite{JL23a}, \cite{JL23b}, \cite{JL25}.  Now we give an informal description of the proof of Theorem $\ref{thm:main}$. 

\bigskip

We first prove the Soliton Resolution Conjecture along a sequence of time in Section ~\ref{sec:3}. Since the Ginzburg-Landau equation ~\eqref{eqn:GL} behaves like the heat equation in energy, we establish some lemmas that are analogous to those in \cite{JL23b}. For solutions that blow up at finite time, the first goal is to identify the behavior of the solution at regular points where the energy doesn't concentrate. We follow the method in \cite{Stru85}, \cite{Qin95} to obtain a body map, the only difference is that we need an exterior energy estimate due to the non-compactness of the domain, see ~\eqref{outer-ball-estimate}. To extract the solitons at the singular point, we use the compactness lemma in \cite{Law23}, where the argument was conducted for heat equations but still works for the Ginzburg-Landau equation. Finally we use some monotonicity formulae as in \cite{JL23b} to obtain the resolution in a sequence of time. For global solutions, the proof is similar, except for the control of energy away from the origin, which is obtained through Lemma ~\ref{lem:localized-coercivity} as in \cite{Ary25ra}. 

\bigskip

Once we establish the sequential soliton resolution, we use the method in \cite{JL23b} to continue the argument. To make things more explicit, we introduce the following definitions.

\begin{defn}[Multi-bubble configuration]
\label{def:d}
    Given $M\in \bbN$, $\vec{\lambda}=(\lambda_1,\dots,\lambda_M)\in(0,\infty)^M$, and $\vec{\theta}\in \bbT^M$, a multi-bubble configuration is defined by 

    \begin{align}
        &\calW(\vec{\theta},\vec{\lambda}):= \sum_{j=1}^M e^{i\theta_j} W_{\lambda_j} \mif M\geq 1,\\
        &\calW(\vec{\theta},\vec{\lambda}):=0 \mif M=0.
    \end{align}
    
\end{defn}

The sequential soliton resolution allows us to fix the number of solitons, denoted by $N$, in Definition ~\ref{def:d}. Next we introduce a function to measure the localized distance to the N-bubble configuration.

\begin{defn}(Proximity to a multi-bubble)
    For all \( t \in I_+ \), \( \rho \in (0, \infty) \), and \( K \in \{0, 1, \ldots, N\} \), we define the localized multi-bubble proximity function as
\[
\mathbf{d}_K(t; \rho) := \inf_{\vec{\theta}, \vec{\lambda}} \left( \|u(t) - u^* - \mathcal{W}(\vec{\theta}, \vec{\lambda})\|_{\mathcal{E}(\rho)}^2 + \sum_{j=K}^N \left( \frac{\lambda_j}{\lambda_{j+1}} \right)^{\frac{D-2}{2}} \right)^{\frac{1}{2}},
\]
where $u^*=0$ when $T_+=\infty$, \( \vec{\te} := (\te_{K+1}, \ldots, \te_N) \in \bbT^{N-K} \), \( \vec{\lambda} := (\lambda_{K+1}, \ldots, \lambda_N) \in (0, \infty)^{N-K} \), \( \lambda_K := \rho \) and \( \lambda_{N+1} := \sqrt{T_+ - t} \) when \( T_+ < \infty \) and \( \lambda_{N+1} := \sqrt{t} \) when \( T_+ = \infty \). 

The multi-bubble proximity function is defined by \( \mathbf{d}(t) := \mathbf{d}_0(t; 0) \).
\end{defn}

The sequential soliton resolution implies that there exists a sequence of time $t_n\to T_+$ such that $\lim \limits_{n\to \infty} \bfd(t_n)=0$. Moreover, proving Theorem ~\ref{thm:main} reduces to showing that $$\lim_{t\to T_+} \bfd(t)=0.$$

We argue by contradiction. Suppose that $\bfd(t)$ does not converge to $0$, then $\bfd(t)$ is large in some sense along a certain sequence of time. The key insight in \cite{JL23b} is to analyze the dynamic behavior of the solution $u$ along a sequence of time intervals $[a_n,b_n]$ such that $a_n,b_n\to T_+$ and that $u$ is close to the N-bubble configuration at $a_n$ and far away at $b_n$. During this time the interaction of the solitons dominates the dynamic behavior of the solution, and the question can be reduced into  a finite-dimensional problem since the radial stationary solutions to ~\eqref{eqn:GL} are expressed explicitly through the Aubin-Talenti solution.  Here we use the same definition of collision intervals as in \cite{JL23b} to describe such intervals.

\bigskip

We say that an interval $[a, b]$ is a collision interval with parameters $0<\veps< \eta$ and $N-K$ exterior bubbles for some $1 \le K \le N$, if $\bfd(a) \le \veps$,  $\bfd(b) \ge \eta$, and there exists a curve $r = \rho_K(t)$ such that $\bfd_{K}(t;\rho_K(t))\leq \veps$ for all $t\in [a,b]$; see Definition~\ref{def:collision}.  We now define $K$ to be the smallest non-negative integer for which there exists $\eta>0$, a sequence $\veps_n \to 0$,  and sequences $a_n, b_n \to \infty$, such that $[a_n, b_n]$ are collision intervals with parameters $\veps_n, \eta$ and $N-K$ exterior bubbles and are pairwise disjoint. For these intervals we write $[a_n, b_n] \in \calC_K( \veps_n, \eta)$. See Section~\ref{sec:4.1} for the proof that $K$ is well-defined and $\ge 1$, under the contradiction hypothesis.

\bigskip

Along these intervals, by using modulation technique we can derive differential inequalities for the scales of the bubbles that come into collision. For the $K-$th bubble with scale $\lam_K$ that comes into collision, we show that for $n$ large enough, there exists $[c_n,d_n]\subset[a_n,b_n]$ such that
\begin{equation}
    \label{ineq:1}
    d_n-c_n\gtrsim \frac{\lambda_K(c_n)^2}{n}.
\end{equation}

Similar inequalities have been established in \cite{JL23a},\cite{JL23b},\cite{Ary25ra}. The difficulty here is the presence of discrete spectrum of the linearized operator, see Section ~\ref{sec:2.2}. We cannot treat that directly since the presence of $z$ in \eqref{eqn:GL}  will make it difficult to control the time derivatives of discrete components. We use the argument from \cite{Ary25ra} to overcome this obstacle.  The key insight in this step lies in realizing that the localized energy inequalities result in the propagation of small energy over short time intervals behaving as in the wave equation. 

\bigskip

The compactness lemma ~\ref{lem:compactness} , Lemma ~\ref{lem:delta-to-d}, and Lemma ~\ref{lem:cndn} yield that for $n$ sufficiently large, there exists a subinterval $[c_n,d_n]\in [a_n,b_n]$ such that ~\eqref{ineq:1} and $\inf_{t\in [c_n,d_n]} \lam_K(c_n)^2 \| \p_tu(t)\|_{L^2}^2\gtrsim1$ hold. Moreover, $[c_n,d_n]$ are pairwise disjoint. Thus we get that

\[ 
1\gtrsim\int_0^{T_+} \| \p_tu(t)\|_{L^2}^2 \, \ud t\gtrsim\sum_{n\in \bN}\int_{c_n}^{d_n} \| \p_tu(t)\|_{L^2}^2 \, \ud t\gtrsim \sum_{n\in \bN} \frac{1}{n}=\infty,
\]
which is a contradiction and therefore Theorem ~\ref{thm:main} is proved.

\bigskip

The remainder of this paper is organized as follows. Section~\ref{sec:2} reviews several fundamental properties of equation~\eqref{eqn:GL}. Next, we establish the sequential resolution for~\eqref{eqn:GL} in Section~\ref{sec:3}. In Section~\ref{sec:4}, we introduce the notion of collision interval and use modulation technique to obtain a precise decomposition for the solution in such time intervals. Finally, we derive a contradiction in Section~\ref{sec:5} by examining the dynamics of the scaling parameters, thereby completing the proof.

\subsection{Notation}

We use the following notations:

\begin{itemize}
\item Given a function $\varphi(r)$ and $\lambda > 0$, we denote by $\varphi_\lambda(r) = \frac{1}{\lambda^{\frac{D-2}{2}}} \varphi(\frac{r}{\lam})$, the $\dot{H}^1$-invariant rescaling, and by $\varphi_{\underline{\lam}}(r) = \frac{1}{\lambda^{\frac{D}{2}} }\varphi(\frac{r}{\lambda})$ the $L^2$-invariant rescaling. Furthermore, we set $\Lambda := r\partial_r + \frac{D-2}{2}$ and $\underline{\Lambda} := r\partial_r + \frac{D}{2}$ as the infinitesimal generators of these scaling.
\item Given two functions $f, g \in L^2((0, \infty), r^{D-1}\mathrm{d}r)$, we define their inner product
\[
\la f \mid g \ra :=\Re \int_0^\infty \bar{f}(r)g(r) r^{D-1}\mathrm{d}r.
\]

\item Given $u \in \mathcal{E}$, we define the modified energy density $\tilde{\mathbf{e}}$ and the localized $\mathcal{E}$ norm as follows
\[
\tilde{\mathbf{e}}(u) := |\partial_r u|^2 + \frac{|u|^2}{r^2}, \quad \tilde{E}(u; r_1, r_2) := \|u\|_{\mathcal{E}(r_1, r_2)}^2 := \int_{r_1}^{r_2} \tilde{\mathbf{e}}(u) r^{D-1}\mathrm{d}r. 
\]
By convention, $\mathcal{E}(r_0) := \mathcal{E}(r_0, \infty)$ for $r_0 > 0$. We similarly define the nonlinear energy density and the localized nonlinear energy as follows
\[
\mathbf{e}(u) := \frac{|\partial_r u|^2}{2} - \frac{|u|^{\frac{2D}{D-2}}}{\frac{2D}{D-2}}, \quad E(u; r_1, r_2) := \int_{r_1}^{r_2} \mathbf{e}(u) r^{D-1}\mathrm{d}r,
\]
where $D \geq 3$, and $2^* := \frac{2D}{D-2}$.
\item Throughout the paper, the function $\chi \in C_c^\infty([0, \infty))$ denotes a smooth radial cut-off function, supported on $r \leq 2$ and $\chi \equiv 1$ when $r \leq 1$. Furthermore we denote $\chi_R(r) := \chi(\frac{r}{R})$.
\item The inequality $A \lesssim B$ means that $A \leq CB$ and $A \gtrsim B$ means that $A \geq cB$ for some constants $C, c > 0$ possibly depending on the number of bubbles $N$ and the dimension $D$. We write $A \ll B$ if $\lim_{n \to \infty} \frac{A}{B} = 0$ for two sequences of numbers $\{A_n\}$ and $\{B_n\}$, or if $\lim \limits_{t\to \infty} \frac{A}{B}$ for two functions of time $A(t)$ and $B(t)$ .
\end{itemize}

\subsection{Acknowledgment}  The author would like to thank Prof. Lifeng Zhao for introducing this interesting problem and so many useful discussions and Dr. Shrey Aryan for his helpful advice. 

\section{Preliminaries} \label{sec:2}
\subsection{Local Cauchy Theory}

The solution to $\eqref{eqn:GL-non-radial}$ is defined in the following sense.

\begin{defn}
		A function $u: [0, T) \times \mathbb{R}^d  \to \mathbb{C}$ for $0< T \le \infty$ is a strong $\dot{H}^1$ solution of \eqref{eqn:GL-non-radial} if for all $t\in [0,T)$,
		$u \in \left(C_t^0 \dot{H}^1_x \cap L_{t,x}^\frac{2(d+2)}{d- 2} \right)([0,t] \times \mathbb{R}^d ) $, and obeys the Duhamel formula
		\begin{align}
			u(t) = e^{t z \Delta } u_0 + \int_0^t e^{(t - s) z \Delta } z f(u(s)) \,\mathrm{d}s.
		\end{align}
	\end{defn}
    
We state a proposition about the local well-posed theory for the Ginzburg-Landau equation $\eqref{eqn:GL-non-radial}$ in the energy space. 

\begin{prop}
\label{prop:cauchy}
    Assume \( u_0 \in \dot{H}^1(\mathbb{R}^d) \). Then the following conclusions hold:
\begin{enumerate}
    \item Local existence: There exists a unique, maximal-lifespan solution to $\eqref{eqn:GL-non-radial}$ on \([0, T_{\max}(u_0)) \times \mathbb{R}^d\).
    \item Blow-up criterion: If \( T_{\max} < \infty \), then \( \|u\|_{S([0, T_{\max}(u_0)))} = \infty \), where $\|u\|_S=\|u\|_{L_{t,x}^{\frac{2(D+2)}{D-2}}}$.
    \item Unconditional uniqueness: If \( u_1, u_2 \) are two solutions of $\eqref{eqn:GL-non-radial}$ on \([0, T)\) with \( u_1(0) = u_2(0) \), then \( u_1 = u_2 \).
    \item Small data global existence: There is \( \varepsilon_0 > 0 \) such that if \( \|e^{tz\Delta}u_0\|_{S(\mathbb{R}_+)} \leq \varepsilon_0 \), then the solution \( u \) is global and \( \|u\|_{S(\mathbb{R}_+)} \lesssim \varepsilon_0 \). This holds in particular when \( \|u_0\|_{\dot{H}^1} \) is sufficiently small.
    \item 
    The energy $E(u(t))$ is absolutely continuous and non-increasing as a function of $t\in[0,T_+)$, and for any $t_1\leq t_2\in [0,T_+)$, there holds,
    \begin{align}
    \label{eq:tension-L2}
    E(u(t_2))+\Re z \cdot \int_{t_1}^{t_2} \int_0^{\infty} |\p_t u(t,r)|^2 r^{D-1} \mathrm{d}r\mathrm{d}t=E(u(t_1)).
    \end{align} 
    In particular, for solution uniformly bounded in $\calE$, we have
    \begin{align}
        \int_{0}^{T_+} \int_0^{\infty} |\p_t u(t,r)|^2 r^{D-1} \mathrm{d}r\mathrm{d}t< \infty.
    \end{align} 
    
    \item 
    \begin{align}
    \label{eq:continuous-dependence}
    \|u(t) - v(t)\|_{\dot{H}^1} + t^{(D-2)/4} \|u(t) - v(t)\|_{L^\infty} \leq C \|u_0 - v_0\|_{\dot{H}^1}
    \end{align}  for all $t\in (0,T)$ where $T=\min \{T_+(u_0),T_+(v_0)\}$ and $C$ can be estimated in terms of $\|u_0\|_{\dot{H}^1}$ and $\|v_0\|_{\dot{H}^1}$
\end{enumerate}
\end{prop}

\begin{proof}
    For the first five terms, see Theorem 2.3 in \cite{CCZ23}. The proof of the last term is an analogous of the proof of Theorem 1 in \cite{BC96}.
\end{proof}

\subsection{Linearization Around a Ground State}
\label{sec:2.2}

 We denote $\calL_{\calW}$ the linearization of $z(\Delta u +f(u))$ near $u=\calW(\vec{\theta},\vec{\lam})$. In particular, we write $\calL_{\lam}=\calL_{\calW}$ for $\calW=W_{\lam}$ and 
$\calL=\calL_{1}$.

Firstly, we give an explicit expression of the linearization of $z(\Delta u +f(u))$

We will write indifferently $ f = f_1 + if_2 $ or $  f = \begin{pmatrix} f_1 \\ f_2 \end{pmatrix} $  for a complex valued function $  f $ with real part $f_1$ and imaginary part $f_2$.

Recall that for $u\in \mathbb{C}$ we denote $f(u)=|u|^{\frac{4}{D-2}}u$ and $F(u)=\frac{D-2}{2D}|u|^{\frac{2D}{D-2}}$.  For $u\in \mathbb{C}$ we define the $\mathbb{R}$-linear operator $f'(u):\mathbb{C}\to \mathbb{C}$ by the following  formula:
\[
f'(u)g:=  \begin{cases}
    |u|^\frac{4}{N-2}\Big(g + \frac{4}{N-2}u\Re(u^{-1}g)\Big) & u\neq 0, \\
    0 & u=0.
\end{cases} 
\]

With this notation the linearization of $z(\Delta u +f(u))$ near $u$ can be expressed as $L \veps = z\Delta \veps +zf'(u)\veps$.

We will express  $\calL$, the linearization of $z(\Delta u +f(u))$ near $u=W$ in the matrix form. We introduce the following notation: 
\[
V^+ := -pW^{p-1}, \quad V^- := -W^{p-1}, \quad L^+ := -\Delta + V^+, \quad L^- := -\Delta + V^-.
\]
Then for $g=g_1+ig_2 \in \calE $, 
\[
\calL g = \begin{pmatrix} a & -b \\ b & a \end{pmatrix} \begin{pmatrix} -L_+ & 0 \\ 0 & -L_- \end{pmatrix} \begin{pmatrix} g_1 \\ g_2 \end{pmatrix}.
\]

It is known that for all $g \in \cE$ there holds $\la g, L^- g\ra \geq 0$ and $\ker L^- = \spn(W)$.
The operator $L^+$ has one simple strictly negative eigenvalue that we denote $-\ka^2<0$( where $\ka>0$) and, 
restricting to radially symmetric functions, $\ker L^+ = \spn(\Lambda W)$. Denote $\calY$ as the normalized (in $L^2$) eigenfunction associated to the eigenvalue $-\ka^2$

One can show that there exist real functions $\cY^{(1)}, \cY^{(2)} \in \cS$ and a real number $\nu > 0$ such that
\begin{equation}
  \label{eq:Y1Y2}
  L^+ \cY^{(1)} = -\nu \cY^{(2)}, \qquad L^- \cY^{(2)} = \nu \cY^{(1)}.
\end{equation}
(the proof given in   for $N = 5$ works in any dimension $N \geq 5$).

\subsection{Multi-Bubble Configuration}
Next, we recall some facts about solutions near multi-bubble configurations. Similar results for real functions were treated in  \cite{JL23b}

Fix any non-negative function $\calZ_1,\calZ_2\in C_c^{\infty}((0,\infty))$ such that the following holds
\begin{align}\label{aux}
    \langle \mathcal{Z}_1 \mid \Lambda W\rangle &> 0, \qquad \langle \mathcal{Z}_1 \mid \mathcal{Y}\rangle = 0, \\
    \langle \mathcal{Z}_2 \mid W\rangle &> 0.
\end{align}

We define the following proximity function to measure how much we deviate from a multi-bubble configuration. 
\begin{defn}
\begin{align}
\label{eq:d-def} 
\bfd_M(  v) := \inf_{\vec {\theta}, \vec \lam}  \Big( \|  v -  \calW( \vec {\theta}, \vec \lam) \|_{\E}^2 + \sum_{j =1}^{M-1} \Big( \frac{\lam_{j}}{\lam_{j+1}} \Big)^{\frac{D-2}{2}} \Big)^{\frac{1}{2}}.
\end{align} 
where the infimum is taken over all vectors $\vec \lam = (\lam_1, \dots, \lam_M) \in (0, \infty)^M$ and all $\vec \theta = \{ \theta_1, \dots, \theta_M\} \in \bbT^M$. 
\end{defn} 

We now state a static modulation lemma in the same spirit of  \cite{JL23a}.

\begin{lem}\label{lem:static-mod} Let $M \in \N$. 
There exists $\eta, C>0$ with the following properties. Let $\te>0$,  and let $ v \in  \calE$   be such that 
\begin{align}
\label{eq:v-M-bub} 
\bfd_M(  v)  \le \eta, \mand E(  v) \le ME(W) + \te^2. 
\end{align} 
Then, there exists a unique choice of $\vec \lam = ( \lam_1, \dots, \lam_M) \in  (0, \infty)^M$, $\vec \theta = \{ \theta_1, \dots, \theta_M\} \in \bbT^M$, and $g \in  \calE$, such that
\begin{align}
\label{eq:v-decomp} 
    v &=   \calW(  \vec \theta, \vec \lam) +  g, \quad 
   0  = \La i e^{i\theta_j} \calZ_{1,\lam_j}   \mid g\Ra=\La e^{i\theta_j} \calZ_{2,\lam_j}   \mid g\Ra  \  , \quad \forall j = 1, \dots, M, 
\end{align}
   along with the estimates, 
\begin{align}
\label{eq:g-bound-0}
\bfd_M(  v)^2 &\le \|   g \|_{\E}^2  + \sum_{j =1}^{M-1} \Big( \frac{\lam_{j}}{\lam_{j+1}} \Big)^{\frac{D-2}{2}}  \le C \bfd_M(  v)^2.
\end{align} 

\end{lem}

\begin{proof}

The proof is a standard argument and is very similar to \cite{JL25}. 

\end{proof}

Furthermore, similar to Lemma 2.20 in \cite{JL23a}, if a function $w$ is close to two different multi-bubble configurations then the scales of those two configurations and the phases of those two configurations are also the same up to a small constant. 
\begin{lem}\label{lem:bub-config} There exists $\eta>0$ sufficiently small with the following property. Let  $M, L \in \N$,  $\vec\theta \in \bbT^M, \vec \sigma \in \bbT^L$, $\vec \lam \in (0, \infty)^M, \vec \mu \in (0, \infty)^L$,   and $ w\in \calE $ satisfying, 
 \begin{align} 
 \| w  -  \calW(  \vec \theta, \vec \lam)\|_{\E}^2  + \sum_{j =1}^{M-1} \Big(\frac{\lam_j}{\lam_{j+1}} \Big)^{\frac{D-2}{2}} &\le \eta,  \label{eq:M-bub} \\
 \| w  -  \calW( \vec \sigma , \vec \mu)\|_{\E}^2 +  \sum_{j =1}^{L-1} \Big(\frac{\mu_j}{\mu_{j+1}} \Big)^{\frac{D-2}{2}} &\le \eta. \label{eq:L-bub} 
 \end{align} 
 Then, $M = L$. Moreover, for every $\delta>0$ the number $\eta>0$ above can be chosen small enough so that 
\begin{align}
\label{eq:lam-mu-close} 
\max_{j = 1, \dots M} \left| \frac{\lam_j}{\mu_j} - 1 \right| + \max_{j = 1, \dots M} |\theta_j-\sigma_j|  \leq  \delta.
\end{align} 
\end{lem}

\subsection{Localized Energy Inequalities and Energy Trapping}

In this subsection we record here some localized energy inequalities.

\begin{lem}
\label{lem:localized-energy}    

    Let \( I \subset [0, \infty) \) be a time interval, and let \( \phi : I \times (0, \infty) \to [0, \infty) \) be a smooth function. Let \( u(t) \in \mathcal{E} \) be a solution to $\eqref{eqn:GL}$ on \( I \). Then, for any \( t_1, t_2 \in I \) with \( t_1 < t_2 \) we have
\[
\begin{aligned}
\label{eqn:localized-energy}
& \int_0^\infty \tilde{\mathbf{e}}(u(t_2)) \phi^2 r^{D-1} \mathrm{d}r - \int_0^\infty \tilde{\mathbf{e}}(u(t_1)) \phi^2 r^{D-1} \mathrm{d}r   \\
& = -2 \Re z \cdot \int_{t_1}^{t_2} \int_0^\infty |\partial_t u|^2 \phi^2 r^{D-1} \mathrm{d}r\mathrm{d}t + 2 \Re \int_{t_1}^{t_2} \int_0^\infty |u|^{p-1} u (\partial_t\bar{u}) \phi^2 r^{D-1} \mathrm{d}r\mathrm{d}t \\
& \quad - 4 \Re \int_{t_1}^{t_2} \int_0^\infty (\partial_r u)(\partial_t \bar{u}) \phi \partial_r \phi \, r^{D-1} \mathrm{d}r\mathrm{d}t + 2 \Re \int_{t_1}^{t_2} \int_0^\infty \frac{u \partial_t \bar{u}}{r^2} \phi^2 r^{D-1} \mathrm{d}r\mathrm{d}t \\
& \quad + 2 \int_{t_1}^{t_2} \int_0^\infty \tilde{\mathbf{e}}(u(t)) \phi \partial_t \phi \, r^{D-1} \mathrm{d}r\mathrm{d}t.
\end{aligned}
\]

\end{lem}

\begin{proof}

    Integrating by parts, we have

    \begin{align}   
    & \int_0^\infty \tilde{\mathbf{e}}(u(t_2)) \phi^2 r^{D-1} \mathrm{d}r - \int_0^\infty \tilde{\mathbf{e}}(u(t_1)) \phi^2 r^{D-1} \mathrm{d}r   = \int_{t_1}^{t_2}  \p_t(\tilde{\mathbf{e}}(u(t)) \phi^2) r^{D-1} \mathrm{d}r \mathrm{d}t \\ 
    & =  \int_{t_1}^{t_2}  ( \p_ru \p_t\p_r \bar{u} +\p_r\bar{u} \p_t\p_r u +\frac{u\p_t\bar{u}+\bar{u}\p_tu }{r^2} )  \phi^2 r^{D-1} \mathrm{d}r \mathrm{d}t+ 2\int_{t_1}^{t_2}  \tilde{\mathbf{e}}(u(t)) \phi \p_t \phi r^{D-1} \mathrm{d}r \mathrm{d}t \\
    & =-2 \Re z \cdot \int_{t_1}^{t_2} \int_0^\infty |\partial_t u|^2 \phi^2 r^{D-1} \mathrm{d}r\mathrm{d}t + 2 \Re \int_{t_1}^{t_2} \int_0^\infty |u|^{p-1} u (\partial_t\bar{u}) \phi^2 r^{D-1} \mathrm{d}r\mathrm{d}t \\
    & \quad - 4 \Re \int_{t_1}^{t_2} \int_0^\infty (\partial_r u)(\partial_t \bar{u}) \phi \partial_r \phi \, r^{D-1} \mathrm{d}r\mathrm{d}t + 2 \Re \int_{t_1}^{t_2} \int_0^\infty \frac{u \partial_t \bar{u}}{r^2} \phi^2 r^{D-1} \mathrm{d}r\mathrm{d}t \\
    &+2\int_{t_1}^{t_2}  \tilde{\mathbf{e}}(u(t)) \phi \p_t \phi r^{D-1} \mathrm{d}r \mathrm{d}t .
    \end{align}

\end{proof}

\begin{rem}
    Through the proof we can see that the monotone formulae used in the proof for nonlinear heat equation also work for Ginzburg-Landau equation, since the main term remains unchanged up to a positive constant multiple $\Re z$ and no new terms are produced.
\end{rem}

Next we record many lemmas that were proved in \cite{Ary25ra}. The proofs work in the complex-valued case.

\begin{lem}[Radial Sobolev embedding]
    
    Let $v\in \calE$, then for $R>0$ we have 
    \begin{equation}
    \label{eqn:r-sobolev}
    |v(R)|\leq \frac{\sqrt{2}}{R^{\frac{D-2}{2}}} \| v\|_{\calE(R)}
    \end{equation}

\end{lem}

\begin{lem}[Coercivity of nonlinear energy for small \( \mathcal{E} \)-norm]
\label{lem:basic-trapping}
   There exist constants \( \delta, C > 0 \) with the following properties. Let \( v \in \mathcal{E} \) be such that \( \|v\|_{\mathcal{E}} \leq \delta \). Then
\[
E(v) \geq C\|v\|_{\mathcal{E}}^2.
\]
\end{lem}

\begin{lem}[Trapping on Tails]
    \label{lem:localized-coercivity}
    There exist constants \( \delta, C > 0 \) with the following property. Let \( v \in \mathcal{E} \) be such that for any \( R > 0 \), \( \|v\|_{\mathcal{E}(R)} \leq \delta \). Then
\[
E(v; R, \infty) \geq C\|v\|_{\mathcal{E}(R)}^2.
\]
\end{lem}

In the remainder of this section we record the properties of short time propagation for Ginzburg-Landau equations. The proofs are the same as those in \cite{Ary25ra}.

\begin{lem}[Propagation of small localized \( \mathcal{E} \) norm]
    There exist constants \( \delta, C > 0 \) with the following property. Let \( u(t) \in \mathcal{E} \) be a solution to $\eqref{eqn:GL}$ with initial data \( u(0) = u_0 \in \mathcal{E} \) on the time interval \( I = [0, T_+(u_0)) \). Let \( 0 < r_1 < r_2 < \infty \). Suppose that
\[
\|u_0\|_{\mathcal{E}(r_1/2, 2r_2)} \leq \delta, \quad \sup_{t \in J} \|u(t)\|_{\mathcal{E}} \lesssim 1 
\]
where \( J = I \cap [0, \delta r_1^2] \). Then,
\[
\sup_{t \in J} \|u(t)\|_{\mathcal{E}(r_1, r_2)} \leq C\delta,
\]
\end{lem}

\begin{lem}[Short time evolution close to \( W \)]

    Let  $\theta\in \bbT$. There exist \( \delta_0 > 0 \) and a function \( \varepsilon_0 : [0, \delta_0] \to [0, \infty) \) with \( \varepsilon_0(\delta) \to 0 \) as \( \delta \to 0 \) with the following properties. Let \( v_0 \in \mathcal{E} \) and let \( v(t) \) denote the unique solution to $\eqref{eqn:GL}$ with \( v_0(0) = v_0 \). Let \( \mu_0, T_0 > 0 \) and suppose that
\[
\|v_0 - e^{i\theta}W_{\mu_0}\|_{\mathcal{E}} + \frac{T_0^2}{\mu_0} = \delta \leq \delta_0.
\]
Then, \( T_0 < T_+(v_0) \) and
\[
\sup_{t \in [0,T_0]} \|v(t) - e^{i\theta} W_{\mu_0}\|_{\mathcal{E}} < \varepsilon_0(\delta).
\]
\end{lem}

\begin{lem}[Localized short time evolution close to \( W \)]
    Let $\theta\in \bbT$. There exist \( \delta_0 > 0 \) and a function \( \varepsilon_0 : [0, \delta_0] \to [0, \infty) \) with \( \varepsilon_0(\delta) \to 0 \) as \( \delta \to 0 \) with the following properties.

Let \( u_0 \in \mathcal{E}, T_0 < T_+(u_0), T_0 \leq \delta_0 r_1^2 \). Let \( u(t) \) denote the unique solution to $\eqref{eqn:GL}$ with \( u_0(0) = u_0 \) and \( \sup_{t \in [0,T_0]} \|u(t)\|_{\mathcal{E}} \lesssim 1 \). Consider, \( \mu_0 > 0, 0 < r_1 < r_2 < \infty \) and suppose that
\[
\|u_0 - e^{i\theta} W_{\mu_0}\|_{\mathcal{E}(r_1/2, 2r_2)} + \frac{T_0^2}{\mu_0} = \delta \leq \delta_0.
\]

Then,
\[
\|u(t) - e^{i\theta} W_{\mu_0}\|_{\mathcal{E}(r_1, r_2)} < \varepsilon_0(\delta)
\]
for all \( t \in (0, T_0] \).
\end{lem}

\begin{lem}
\label{lem:un-seq}
     If $\theta_n\in \bbT$, $0 < t_n^2 < r_n \ll \mu_n \ll R_n $ and \( u_n \) a sequence of solutions of $\eqref{eqn:GL}$ such that \( u_n(t) \) is defined for \( t \in [0, t_n] \) and
\[
\lim_{n \to \infty} \|u_n(0) - e^{i\theta_n} W_{\mu_n}\|_{\mathcal{E}(r_n/2, 2R_n)} = 0,
\]
then
\[
\lim_{n \to \infty} \sup_{t \in [0,t_n]} \|u_n(t) - e^{i\theta_n} W_{\mu_n}\|_{\mathcal{E}(r_n, R_n)} = 0.
\]
\end{lem}

\section{The Sequential Decomposition}\label{sec:3}

\subsection{Identification of the Body Map}

In this subsection we aim to establish the proposition of the existence of the body map in case of finite time blow-up.

\begin{prop} [Body map in case of finite time blow-up]
\label{body-map}

 Let $u_0\in \calE$ and let $u(t)$ be a solution to $\eqref{eqn:GL}$ with $u(0)=u_0$. Suppose that $T_+=T_+(u_0)<\infty$ and $\sup_{t\in[0,T_+)}\left \| u(t) \right\|_{\calE} <\infty$.
 Then, there exists a mapping $u^* \in \mathcal{E}$ such that for any $r_0>0$,

\begin{align}\label{eqn:con-of-radiation}
\lim _{t \to T_*}\left\|u(t)-u^*\right\|_{\mathcal{E}\left(r \geq r_0\right)}=0 .    
\end{align}

Moreover, there exist positive constants $L,K>0$ such that for each $r_0\in (0,\infty]$,

\begin{align}
    \label{eqn:con-norm}
    \lim_{t\to T_+} \int_{0}^{r_0} |\p_r u(t)|^2r^{D-1} \mathrm{d}r =L +\int_{0}^{r_0} |\p_r u^*|^2r^{D-1} \mathrm{d}r, \\
    \lim_{t\to T_+} \int_{0}^{r_0} | u(t)|^{\frac{2D}{D-2}}r^{D-1} \mathrm{d}r =K +\int_{0}^{r_0} | u^*|^{\frac{2D}{D-2}}r^{D-1} \mathrm{d}r.
\end{align}
\end{prop}

\begin{proof}

We first prove the existence of $u^*$ as a weak limit in  $\dot{H}^1$ . It is equivalent to showing that all weak limits of sequences $\{u(t_n)\}_n$, where $t_n\to T_+^-$, coincide. 

We argue through contradiction. Suppose that there exist two sequences $\{u(t_{n,1})\}_n, \{u(t_{n,2})\}_n$ such that their weak limits, say $u_1,u_2$, do not coincide, then  there exists a Schwartz function $\varphi$ that satisfies $ \langle\varphi\mid u_1\rangle_{\dot{H}^1}\neq \langle\varphi\mid u_2\rangle_{\dot{H}^1} $ . However 
\begin{align}
    &|\int_{\bbR^D} [u(t_{2,n},x)- u(t_{1,n},x)] \varphi(x) \rm{d} x| =  |\int_{t_{1,n}}^{t_{2,n}}\int_{\bbR^D} \p_s u(s,x)\varphi(x) \rm{d} x \rm{d} s| \\
    &\leq \int_{t_{1,n}}^{t_{2,n}} \| \p_s u(s,x)\|_{L_x^2} \| \varphi(x)\|_{L_x^2} \rm{d} x \rm{d} s \lesssim ( \int_{t_{1,n}}^{t_{2,n}} \| \p_s u(s,x)\|_{L_x^2}^2 \rm{d} s)^{\frac{1}{2}}T_+^{\frac{1}{2}} .
\end{align}

This leads to a contradiction since $T_+<\infty$ and $ \int_{0}^{T_+} \| \p_s u(s,x)\|_{L_x^2}^2 \rm{d} s<\infty$.

To improve the weak convergence to strong convergence , we establish a lemma that is analogous to the standard parabolic regularity theory.

\begin{lem} 
\label{lem:parabolic-regularity}
    For any \( k \in \mathbb{N} \), there are \( \varepsilon_0 > 0 \) ,  and \( C \) such that: for any $r>0$ and $x\in \bbR^d$,   if \( u \) is a solution of equation $\eqref{eqn:GL}$ in \( Q_r(x,T_+) \), where \( Q_r(x,T_+)  := B_r(x) \times [T_+-r^2,T_+) \) for \( r > 0 \), and satisfies
\[
\varepsilon := \|u\|_{L_t^\infty(\dot{H}_x^1 \cap L_x^{\frac{2D}{D-2}})(Q_r(x,T_+) )} < \varepsilon_0,
\]
then \( u \) is smooth on \( \overline{Q_{\frac{r}{2}}}(x,T_+) \) with bounds
\[
\max_{\overline{Q_{\frac{r}{2}}(x,T_+)}} |D^k u| \leq C\varepsilon.
\]
\end{lem}

\begin{proof}[Proof of Lemma ~\ref{lem:parabolic-regularity} ]
    By scaling we may assume $r=1$. The proof is the same as the proof of Lemma 5.1 in \cite{GR18}.
\end{proof}

Define the regular subset as follows:

\begin{align}
\label{def:regular-set}
\calR=\{x\in \bbR^D: \forall \veps>0, \exists R,\ \forall t\in [0,T_+),   \int_{B_R(x)} |u(t,y)|^{\frac{2D}{D-2}} + |\nabla u(t,y)|^2+\frac{|u(t,y)|^2}{|y|^2} \mathrm{d} y \leq \veps .\}
\end{align}

Since $u$ is radially symmetric and is type II, we have $\calR=\bbR \backslash\{0\}$. 
For any $x\in \calR$, there exists $r>0$ such that $u\in C^2(\overline{Q_{\frac{r}{2}(x,T_+)}})$. Applying Arzel\`a-Ascoli theorem, $u(t)\to u^*$ in $C^1(B(x,\frac{r}{2}))$ as $t\to T_+$. Then we have $u(t)\to   u^*$ in $\dot{H}^1_{loc}(\bbR^D\backslash\{0\})$. 

Taking the smooth function in $\eqref{eqn:localized-energy}$ as $\phi(r)=1-\chi(\frac{r}{R})$, where $R$ is large enough , $t_1$ as $t$ and $t_2$ as $0$. We get a uniform exterior energy estimate
\begin{align}
\label{outer-ball-estimate}
   &  \| u(t) \|_{\calE(r\geq 2R)}^2 \leq \| u(0) \|_{\calE(r\geq R)}^2 \\
   &\quad + 4 \left(\int_{t_1}^{t_2} \int_0^\infty |\partial_t u|^2 |\partial_r \phi|^2 \phi^2 r^{D-1} \, \mathrm{d}r\mathrm{d}t \right)^{\frac{1}{2}} 
   \left(\int_{t_1}^{t_2} \int_0^\infty |\partial_r u|^2  r^{D-1} \, \mathrm{d}r\mathrm{d}t \right)^{\frac{1}{2}} 
   \\
&\quad + 2 \left( \int_{t_1}^{t_2} \int_0^\infty |u|^{2p} \phi^2 r^{D-1} \, \mathrm{d}r\mathrm{d}t \right)^{1/2} \left( \int_{t_1}^{t_2} \int_0^\infty (\partial_t u)^2 \phi^2 r^{D-1} \, \mathrm{d}r\mathrm{d}t \right)^{1/2} \\
&\quad + 2 \left( \int_{t_1}^{t_2} \int_0^\infty \frac{|u|^2}{r^4} \phi^2 r^{D-1} \, \mathrm{d}r\mathrm{d}t \right)^{1/2} \left( \int_{t_1}^{t_2} \int_0^\infty (\partial_t u)^2 \phi^2 r^{D-1} \, \mathrm{d}r\mathrm{d}t \right)^{1/2} \\
&\lesssim  \| u(0) \|_{\calE(r\geq R)} ^2+ \frac{1}{R},
\end{align}

which ensures the global convergence in $\dot{H}(r\geq r_0)$.

To see $\eqref{eqn:con-norm}$, we argue as in the proof of Proposition 2.1 in \cite{Qin95}. Denote

\begin{align}
\label{eq:sup-inf}
\limsup_{t \to T_+} \int_0^{r_0} |\partial_r u(t)|^2 r^{D-1}\mathrm{d}r &= \lim_{t_i \to T_+} \int_0^{r_0} |\partial_r u(t_i)|^2 r^{D-1}\mathrm{d}r = M + \int_0^{r_0} |\partial_r u^*|^2 r^{D-1}\mathrm{d}r, \\
\liminf_{t \to T_+} \int_0^{r_0} |\partial_r u(t)|^2 r^{D-1}\mathrm{d}r &= \lim_{s_i \to T_+} \int_0^{r_0} |\partial_r u(s_i)|^2 r^{D-1}\mathrm{d}r = m + \int_0^{r_0} |\partial_r u^*|^2 r^{D-1}\mathrm{d}r.
\end{align}

We will first show that \( M = m \). Take \( \lambda_j = \frac{r_0}{j} \). Then we have

\begin{equation}
\label{eq:a}
\begin{split}
&\int_0^{\lambda_j} |\partial_r u(t)|^2 r^{D-1}\mathrm{d}r 
= \int_0^{r_0} |\partial_r u(t)|^2 r^{D-1}\mathrm{d}r - \int_{\lambda_j}^{r_0} |\partial_r u(t)|^2 r^{D-1}\mathrm{d}r \\
&\geq \int_0^{r_0} |\partial_r u(t)|^2 r^{D-1}\mathrm{d}r - \int_{\lambda_j}^{r_0} \left||\partial_r u(t)|^2 - |\partial_r u^*|^2\right| r^{D-1}\mathrm{d}r 
- \int_{\lambda_j}^{r_0} |\partial_r u^*|^2 r^{D-1}\mathrm{d}r.
\end{split}
\end{equation}

Then since \( u(t) \to u^* \) in $\calE(r\geq R)$ for each $R>0$ as \( t \to T_+ \), there exists a subsequence \( \{t_{i_j}\} \) such that
\[
\int_{\lambda_j} ^{r_0} \left||\partial_r u(t_{i_j})|^2 - |\partial_r u^*|^2\right| r^{D-1}\mathrm{d}r \to 0, \quad \text{as } j \to \infty,
\]

which combined with $\eqref{eq:sup-inf}$ and $\eqref{eq:a}$ implies
\[
\begin{aligned}
\int_0^{\lambda_j} |\partial_r u(t_{i_j})|^2 r^{D-1}\mathrm{d}r &\geq M + \int_0^{r_0} |\partial_r u^*| r^{D-1}\mathrm{d}r - \int_{\lambda_j}^{r_0} |\partial_r u^*|^2 r^{D-1}\mathrm{d}r - o_j(1) \\
&\geq M + \int_0^{\lambda_j} |\partial_r u^*| r^{D-1}\mathrm{d}r - o_j(1),
\end{aligned}
\]
where \( o_j(1) \to 0 \) as \( j \to \infty \). Thus setting \( t_j = t_{i_j} \) we get
\begin{align}
\label{eq:lower-in}
\int_0^{\lambda_j} |\partial_r u(t_j)|^2 r^{D-1}\mathrm{d}r \geq M - o_j(1). 
\end{align} 

Up to subsequences we can arrange that \( s_j \leq t_j \leq s_{j+1} \leq t_{j+1} \). Then, 
for  $\delta>0$ fixed and $j$ sufficiently large, we choose a cut-off function $\xi_j\in C^{\infty}_c(\bbR^d)$ such that $0\leq \xi_j\leq 1$, $ \xi|_{|x|\leq \lambda_j}=1$, $\xi_j|_{|x|\geq \delta}=0$ and $|\nabla\xi_j(x)|\lesssim\frac{1}{|\delta-\lam_j|}$.

Calculating $\int_0^{\infty} |\p_ru(s_{j+1})|^2r^{D-1} \mathrm{d}r- \int_0^{\infty} |\p_ru(t_{j})|^2r^{D-1} \mathrm{d}r$ as in Lemma $\ref{lem:localized-energy}$ and 
estimating the nonlinear term in \( L^\infty \) on \( (s_j, T_+) \) using $\eqref{eq:continuous-dependence}$ in Proposition $\ref{prop:cauchy}$ with $v_0=0$ we get

\begin{align}
\label{eq:lower-in-delta}
\int_0^\delta |\partial_r u(s_{j+1})|^2 r^{D-1}\mathrm{d}r &\geq \int_0^{\lambda_j} |\partial_r u(t_j)|^2 r^{D-1}\mathrm{d}r - 2\int_{s_j}^{T_+} \|\partial_t u(t)\|_{L^2}^2 \mathrm{d}t \\
&\quad - C_1\sqrt{(T_+ - s_j)\lambda_j^D} \sqrt{\int_{s_j}^{T_+} \|\partial_t u\|_{L^2}^2\mathrm{d}t} - C_2\frac{\sqrt{T_+ - s_j}}{|\delta - \lambda_j|} \sqrt{\int_{s_j}^{T_+} \|\partial_t u\|_{L^2}^2\mathrm{d}t},
\end{align} 

where the constants \( C_1, C_2 > 0 \) depend on \( \sup_{t \in [0, T_+)} \|u(t)\|_{\mathcal{E}} \). Consequently, sending \( j \to \infty \) in $\eqref{eq:lower-in-delta}$ and using $\eqref{eq:lower-in}$ we get

\begin{align}
\lim_{s_j \to T_+} \int_0^{r_0} |\partial_r u(s_j)|^2 r^{D-1}\mathrm{d}r &= \lim_{s_j \to T_+} \int_0^{r_0} |\partial_r u(s_j)|^2 r^{D-1}\mathrm{d}r + \int_0^\delta |\partial_r u(s_j)|^2 r^{D-1}\mathrm{d}r \\
&\geq M + \int_\delta^{r_0} |\partial_r u^*|^2 r^{D-1}\mathrm{d}r.
\end{align}

Taking \( \delta \to 0 \), we can conclude that \( m = M \). Thus, the limit \( L = M \) exists in $\eqref{eqn:con-norm}$. Furthermore, \( L > 0 \) due to the definition of regular set and that $0\notin \calR$. We can argue in the same way and show that there exists \( K > 0 \) such that \( \lim_{t \to T_+} \int_0^{r_0} |u(t)|^{2^*} r^{D-1}\mathrm{d}r = K + \int_0^{r_0} |u^*|^{2^*} r^{D-1}\mathrm{d}r \).

\end{proof}

\subsection{Localized Sequential Bubbling}

We define a localized distance function  
\begin{align}
\label{def:delta}
    \delta_R(u)=\inf_{M,\vec{\theta},\vec{\lambda}} \left(\| u-\calW(\vec{\theta},\vec{\lambda} )\|_{\E}^2 +\sum_{j=1}^M (\frac{\lambda_j}{\lambda_{j+1}})^{\frac{D-2}{2}} \right)^{\frac{1}{2}}
\end{align}

where the infimum is taken over all vectors $\vec \lam = (\lam_1, \dots, \lam_M) \in (0, \infty)^M$ and all $\vec \theta = \{ \theta_1, \dots, \theta_M\} \in \bbT^M$, and we use the convention that the last scale $\lambda_{M+1}=R$.

\begin{lem}[Compactness Lemma]
\label{lem:compactness}
     Let \( u_n \in \mathcal{E} \) be a sequence with \( \limsup_{n \to \infty} \|u_n\|_{\mathcal{E}} < \infty \). Let \( \rho_n \in (0, \infty) \) be a sequence and suppose that
\[
\lim_{n \to \infty} \left( \rho_n \|\mathcal{T}(u_n)\|_{L^2} \right) = 0.
\]

Then, there exists a sequence \( R_n \to \infty \) so that, up to passing to a subsequence of the \( u_n \), we have,
\[
\lim_{n \to \infty} \delta_{R_n \rho_n} (u_n) = 0.
\]

The subsequence \( u_n \) can be chosen so that there are fixed $ M \in \mathbb{N} $, a sequence $\vec{\lambda}_n \in (0, \infty)^M$  ,  $\vec{\theta}\in \bbT^M$ and $ C_0 > 0 $with
\[
\lim_{n \to \infty} \left( \|u_n - \mathcal{W}(\vec{\theta}, \vec{\lambda}_n)\|_{\mathcal{E}(r \leq R_n \rho_n)}^2 + \sum_{j=1}^{M-1} \left( \frac{\lambda_{n,j}}{\lambda_{n,j+1}} \right)^{\frac{D-2}{2}} \right)^{\frac{1}{2}} = 0,
\]
and,
\[
\lambda_{n,M} \leq C_0 \rho_n, \quad \forall n.
\]
\end{lem}

\begin{rem}
    The fact that $\vec\theta$ we get is independent of $n$ comes from the details in the proof of Proposition 1.3 in \cite{Law23}, which can also be obtained from the compactness of $\bbT^N$ . 
\end{rem}

\begin{proof}
    The proof is  the same as that in \cite{Law23}. 
\end{proof}

\subsection{The Sequential Decomposition for Finite Time Blow-Up Solutions}
\begin{prop}[Sequential decomposition in case of finite time blow-up]\label{prop:finite-seq-dec}
Let $u_0\in \calE$, and let $u(t)$ denote the solution to $\eqref{eqn:GL}$ with initial data $u_0$. Suppose that $T_+(u_0)<\infty$. Then, there exist a function $u^*\in \calE$ , an integer $N\geq 1$, a sequence of time $\{t_n\}_{n\geq 1}$ such that $t_n \to T_+$, a sequence of phases $\vec{\theta_n}=(\theta_{1,n},\dots, \theta_{N,n})\in \bbT^N$, a sequence of scales $\vec{\lam_n}=(\lam_{1,n},\dots, \lam_{N,n})\in (0,\infty)^N$, and an error $g_n$ defined by 
\[
u(t_n)=\sum_{j=1}^N e^{i\theta_{j,n}}W_{\lambda_{j,n}}+u^*+g_n
\]
with the following properties:
\begin{enumerate}[label={(\roman*).}]
    \item  The integer $N\geq 1$ and the weak limit $u^*$ satisfy,
    \begin{align}
    \label{eq:energy-limit} 
\lim_{t \to T_+} E( u(t)) = N E(W)  + E( u^*); 
\end{align} 
    \item for any $0<\al<A$ 
\begin{align}
\label{eq:N-bubbles-bu} 
\lim_{t \to T_+} E( u(t); 0, \al(T_+ - t)^{\frac{1}{2}}) = NE(W), 
\end{align} 
\begin{align}
\label{eq:en-annulus}
\lim_{t\to T_+} E(u(t); \al (T_+ -t)^{\frac{1}{2}}, A(T_+ -t)^{\frac{1}{2}}) = 0,
\end{align} 
\begin{align}
\label{eq:en-ext-bu} 
\lim_{t \to T_+} E( u(t) - u^*; \al(T_+ - t)^{\frac{1}{2}}, \infty) = 0, 
\end{align} 
and there exists $0< T_0 < T_+$ and function $\rho : [T_0, T_+) \to (0,\infty)$ satisfying, 
\begin{equation} \label{eq:radiation} 
\lim_{t \to T_+} \big((\rho(t) / \sqrt{T_+-t}) + \| u(t) -  u^*\|_{\cE(\rho(t))}\big) = 0; 
\end{equation}

    \item the error $g_n$ and the scales $\vec \lam_n$ satisfy, 
\begin{align}\label{eq:d(t_n)} 
\lim_{ n \to \infty} \Big( \| g_n \|_{\E}^2 + \sum_{j =1}^N \Big( \frac{\lam_{n, j}}{\lam_{n, j+1}} \Big)^{\frac{D-2}{2}}  \Big)^{\frac{1}{2}}  = 0 , 
\end{align} 
where here we adopt the convention that $\lam_{n, N+1} := (T_+ - t_n)^{\frac{1}{2}}$. 
    
\end{enumerate}
\end{prop}

\begin{proof}[Proof of Proposition~\ref{prop:finite-seq-dec}]

We follow the arguments by \cite{JL23b} and \cite{Ary25ra}. 
Considering a solution $u(t) \in \E$ to \eqref{eqn:GL} which blows up at time $T_+ \in(0,\infty)$. Using ~\eqref{eq:tension-L2} we can find a sequence $t_n \to T_+$ so that, 
\begin{align}
(T_+ - t_n)^{\frac{1}{2}} \| \calT( u(t_n)) \|_{L^2}  \to 0 \mas n \to \infty. 
\end{align} 
Now we apply Lemma~\ref{lem:compactness} with $\rho_n:= (T_+ - t_n)^{\frac{1}{2}}$ to obtain that there exist $N \ge 0$, $m_0\in \Z$, $\vec \te \in \bbT^N, \vec \lam_n \in (0, \infty)^N$ such that  
\begin{align} \label{eq:decomp-0} 
\lim_{n \to \infty} \Big( \| u(t_n) - \calW(\vec \te, \vec \lam_n) \|_{\E( r \le A (T_+- t_n)^{\frac{1}{2}})}^2 + \sum_{j =1}^{N-1} \Big( \frac{ \lam_{n, j}}{\lam_{n, j+1}} \Big)^{\frac{D-2}{2}} \Big) = 0
\end{align} 
up to a subsequence for each $A>0$, and $\lam_{n, N} \lesssim (T_+- t_n)^{\frac{1}{2}}$ if $N\geq1$. Consider a smooth cutoff function
\begin{align}
\phi &\equiv 1\text{ on } [r_1,r_2),\quad \phi \equiv 0 \text{ on }(0,r_1/2]\cap[2r_2,\infty),\quad 0\leq \phi\leq 1
\end{align} 
where the  parameters $r_1\in(0,\infty)$ and $r_2\in(0,\infty]$ will be chosen appropriately. Define the localized $\E$-norm 
\begin{align}
    \ti{\Theta}_{[r_1,r_2]}(t) = \int \ti{e}(u(t,r)) \phi(r)^2  r^{D-1} \ud  r,\quad \ti{\Theta}^*_{[r_1,r_2]} = \int \ti{e}(u^*(r)) \phi(r)^2 r^{D-1} \udr.
\end{align}
Using Lemma $\ref{lem:localized-energy}$ with $|\p_r \phi|\lesssim\frac{1}{r_1}$, we see that for each $0<s<\tau< T_+$ we have, 
\begin{align}
\label{eq:Theta-ineq} 
&\Big| \ti{\Theta}_{[r_1,r_2]}(\tau) - \ti{\Theta}_{[r_1,r_2]}(s) \Big| \\
&\lesssim  \int_{s}^{\tau} \|\partial_t u\|^2_{L^2}\ud t+ \left(\int_{s}^{\tau}\int_0^{\infty} |u|^{2p} \phi^2 \ud r \ud t\right)^{1/2}  \left(\int_{s}^{\tau}\int_0^{\infty} |\partial_t u|^2 \phi^2 \ud r \ud t\right)^{1/2}  \\
&\quad + \left(\int_{s}^{\tau}\int_0^{\infty} |\partial_r u|^2 \phi^2 (\partial_r \phi)^2 \ud r\ud t\right)^{1/2}   \left(\int_{s}^{\tau}\int_0^{\infty} |\partial_t u|^2 \right)^{1/2} \\
&\quad+ \left(\int_{s}^{\tau} \int_0^{\infty} \frac{|u|^2}{r^4} \phi^2\right)^{1/2}  \left(\int_{s}^{\tau} \int_0^{\infty}|\partial_t u|^2 \phi^2 \ud r\ud t\right)^{1/2}\\
&\lesssim  \int_s^{T_+} \| \p_t u(t) \|_{L^2}^2 \, \ud t  + \sqrt{\int_{s}^{\tau} \int_{r_1}^{r_2} \frac{1}{r^3}\ud r \ud t} \sqrt{\int_{s}^{\tau} \|\partial_t u\|^2_{L^2} \ud t} \\
&\quad + \frac{(\tau-s)^{1/2}}{r_1} \sqrt{\int_{s}^{\tau} \|\partial_t u\|^2_{L^2} \ud t}+\frac{(\tau-s)^{1/2}}{r_1} \sqrt{\int_{s}^{\tau} \|\partial_t u\|^2_{L^2}\ud t}\\
&\lesssim  \int_s^{T_+} \| \p_t u(t) \|_{L^2}^2 \, \ud t  + \frac{(T_+-s)^{1/2}}{r_1} \sqrt{\int_{s}^{T_+} \|\partial_t u\|^2_{L^2} \,\ud t}.
\end{align}
Let $s\to T_+$ we see that $\lim_{s\to T_+} \Tilde{\Theta}_{[r_1,r_2]}(s)$ exists. 

We first prove
\begin{align}\label{eqn:fixed-annulus-estimate}
\lim_{t\to T_+}\ti E(u(t);\alp(T_+-t)^{1/2},r_0) = \ti E(u^*;0,r_0)
\end{align}
for any $r_0 \in (0, +\infty].$ 

We decompose the difference into the following four terms and we will treat them separately.
\begin{align}
&\left| \ti{E}\left(u(s); \alpha(T_+ - s)^{\frac{1}{2}}, r_0\right) - \ti{E}\left(u^*; 0, r_0\right) \right| \\
&\quad \lesssim \left| \ti{\Theta}_{\left[\alpha(T_+ - s)^{\frac{1}{2}}, r_0\right]}(s) - \ti{\Theta}^*_{\left[\alpha(T_+ - s)^{\frac{1}{2}}, r_0\right]} \right| + \ti{E}\left(u(s); \frac{\alpha}{2}(T_+ - s)^{\frac{1}{2}}, \alpha(T_+ - s)^{\frac{1}{2}}\right) \\
&\quad\quad + \ti{E}\left(u^*; 0, \alpha(T_+ - s)^{\frac{1}{2}}\right) + \left| \int_{r_0}^{2r_0} \left( \ti{\mathbf{e}}(u(s)) - \ti{\mathbf{e}}(u^*) \right) \phi^2 r^{D-1} \mathrm{d}r \right|.
\end{align}

The third term tends to zero due to $u^*\in\dot{H}^1$ and Lebesgue convergence theorem, and the last term tends to zero by ~\eqref{eqn:con-of-radiation}. Then we only need to treat the first two terms.

Fix $0<s<T_+$, taking $r_1=\al (T_+-s)^{\frac{1}{2}}$ and $r_2=r_0$ in the estimate ~\eqref{eq:Theta-ineq} yields that for $s<\tau<T_+$,
\begin{align}
    \left| \ti{\Theta}_{\left[\alpha(T_+ - s)^{\frac{1}{2}}, r_0\right]}(s)-\ti{\Theta}_{\left[\alpha(T_+ - s)^{\frac{1}{2}}, r_0\right]}(\tau) \right|\lesssim \int_s^{T_+} \| \p_t u(t) \|_{L^2}^2 \, \ud t  + \frac{1}{\al} \sqrt{\int_{s}^{T_+} \|\partial_t u\|^2_{L^2} \,\ud t}.
\end{align}

Since the right-hand side above is independent of $\tau$ and 
\begin{align}
    \left|\ti{\Theta}_{\left[\alpha(T_+ - s)^{\frac{1}{2}}, r_0\right]}(\tau)-\ti{\Theta}_{\left[\alpha(T_+ - s)^{\frac{1}{2}}, r_0\right]}^*\right| \to 0
\end{align}
as $\tau\to T_+$, the first term tends to zero as $s\to T_+$. Using the same argument we obtain
\begin{align}
    \lim_{s\to T_+}\ti{\Theta}_{\left[\frac{\al}{2}(T_+ - s)^{\frac{1}{2}}, \al(T_+ - s)^{\frac{1}{2}}\right]}(s)=0,
\end{align}
which gives an estimate for the second term. Combining the arguments above we get ~\eqref{eqn:fixed-annulus-estimate}.

Next we prove ~\eqref{eq:radiation}. We only need to prove that
\begin{align}
    \lim_{t\to T_+}\ti E(u(t)-u^*;\al(T_+ - s)^{\frac{1}{2}},\infty)=0,\quad \forall \ \al>0.
\end{align}

We estimate as follows:
\begin{align}
\ti{E}\left(u(\tau) - u^*; \alpha(T_+ - \tau)^{\frac{1}{2}}, \infty\right) &\leq
2\ti{E}\left(u(\tau); \alpha(T_+ - \tau)^{\frac{1}{2}}, r_0\right) + 2\ti{E}\left(u^*; \alpha(T_+ - \tau)^{\frac{1}{2}}, r_0\right) \\&+ \ti{E}\left(u(\tau) - u^*; r_0, \infty\right).
\end{align}

For any $\veps>0$, we take $r_0>0$ sufficiently small so that $2\ti{E}\left(u^*; 0, r_0\right)+ \ti{E}\left(u(\tau) - u^*; r_0, \infty\right)\leq \frac{1}{3}\veps $. Then using ~\eqref{eqn:fixed-annulus-estimate} we can choose $0<\tau_0<T_+$ so that for any $\tau_0<\tau<T_+$, 
\begin{align}
    \ti E(u(t);\alp(T_+-t)^{1/2},r_0)\leq \ti{E}\left(u^*; 0, r_0\right)+\frac{1}{6}\veps\leq\frac{1}{3}\veps.
\end{align}
Therefore
\begin{align}
\ti{E}\left(u(\tau) - u^*; \alpha(T_+ - \tau)^{\frac{1}{2}}, \infty\right) &\leq
2\ti{E}\left(u(\tau); \alpha(T_+ - \tau)^{\frac{1}{2}}, r_0\right) + 2\ti{E}\left(u^*; 0, r_0\right) \\&+ \ti{E}\left(u(\tau) - u^*; r_0, \infty\right)\leq \veps,
\end{align}
which implies ~\eqref{eq:radiation}.

Now we prove ~\eqref{eq:en-ext-bu}. We have
\begin{align}
    E( u(t) - u^*; \al(T_+ - t)^{\frac{1}{2}}, \infty) &=\frac{1}{2}\int_{\al(T_+ - t)^{\frac{1}{2}}}^{\infty} |\p_r(u(t) - u^*)|^2 r^{D-1}\,\ud r\\&-\frac{D-2}{2D}\int_{\al(T_+ - t)^{\frac{1}{2}}}^{\infty} | u(t) - u^*|^{\frac{2D}{D-2}}r^{D-1}\,\ud r.
\end{align}
The first term tends to zero due to ~\eqref{eq:radiation}. For the second term, using radial Sobolev embedding we have that
\begin{align}
    \int_{\al(T_+ - t)^{\frac{1}{2}}}^{\infty} | u(t) - u^*|^{\frac{2D}{D-2}}r^{D-1}\,\ud r&=\int_{\al(T_+ - t)^{\frac{1}{2}}}^{\infty} r^2| u(t) - u^*|^{\frac{4}{D-2}}\frac{| u(t) - u^*|^2}{r^2}r^{D-1}\,\ud r\\ &\lesssim
    \|u(t)-u^*\|_{\E(\al(T_+ - t)^{\frac{1}{2}})}^{\frac{2D}{D-2}} \to 0 \quad \text{as}\quad t\to T_+,
\end{align}
which proves ~\eqref{eq:en-ext-bu}.

The proof of $\eqref{eq:en-annulus}$ follows directly from that of ~\eqref{eq:en-ext-bu}, so we omit the details.

Next we prove ~\eqref{eq:N-bubbles-bu}. $\eqref{eq:en-annulus}$ implies that the scaling parameter $\lam_{n,N}$ appearing in the sequential decomposition ~\eqref{eq:decomp-0} must satisfy $\lam_{n,N}\ll (T_+-t_n)^{\frac{1}{2}} $, therefore ~\eqref{eq:N-bubbles-bu} holds for a sequence of time. So it suffices to prove that $\lim \limits_{t\to T_+}E( u(t); 0, \al(T_+ - t)^{\frac{1}{2}})$ exists. Using ~\eqref{eqn:con-norm}, we have
\begin{align}
    L+\int_0^{\infty}|\p_r u^*|^2 r^{D-1}\,\ud r&=\lim_{t\to T_+}\int_0^{\infty}|\p_r u(t)|^2 r^{D-1}\,\ud r\\
    &=\lim_{t\to T_+}\int_0^{\al (T_+-t)^{\frac{1}{2}}}|\p_r u(t)|^2 r^{D-1}\,\ud r+\lim_{t\to T_+}\int_{\al (T_+-t)^{\frac{1}{2}}}^{\infty}|\p_r u(t)|^2 r^{D-1}\,\ud r\\
    &=\lim_{t\to T_+}\int_0^{\al (T_+-t)^{\frac{1}{2}}}|\p_r u(t)|^2 r^{D-1}\,\ud r+\lim_{t\to T_+}\int_{\al (T_+-t)^{\frac{1}{2}}}^{\infty}|\p_r u^*|^2 r^{D-1}\,\ud r\\
    &+\lim_{t\to T_+}\int_{\al (T_+-t)^{\frac{1}{2}}}^{\infty}|\p_r (u(t)-u^*)|^2 r^{D-1}\,\ud r\\
    &+O\left(\lim_{t\to T_+}\|u(t)-u^*\|_{\E(\al (T_+-t)^{\frac{1}{2}})}^2 \right)\\
    &=\lim_{t\to T_+}\int_0^{\al (T_+-t)^{\frac{1}{2}}}|\p_r u(t)|^2 r^{D-1}\,\ud r+\int_0^{\infty}|\p_r u^*|^2 r^{D-1}\,\ud r,
\end{align}
therefore $\lim \limits_{t\to T_+}\int_0^{\al (T_+-t)^{\frac{1}{2}}}|\p_r u(t)|^2 r^{D-1}\,\ud r$ exists and equals to $N\|W\|_{\dot{H}^1}^2$ due to the decomposition ~\eqref{eq:decomp-0}. Since $L>0$, we have $N\geq 1$. Using the similar argument we get $\lim \limits_{t\to T_+}\int_0^{\al (T_+-t)^{\frac{1}{2}}}|u(t)|^{\frac{2D}{D-2}} r^{D-1}\,\ud r$ exists and equals to  $N\|W\|_{L^{\frac{2D}{D-2}}}^{\frac{2D}{D-2}}$. The arguments above yields the existence of $\lim \limits_{t\to T_+}E( u(t); 0, \al(T_+ - t)^{\frac{1}{2}})$, thus proves ~\eqref{eq:N-bubbles-bu}. We also note that \eqref{eq:d(t_n)} is covered by the previous proof.

\eqref{eq:energy-limit}  follows from ~\eqref{eq:N-bubbles-bu}, $\eqref{eq:en-annulus}$ and ~\eqref{eq:en-ext-bu}.

\end{proof}

\subsection{The Sequential Decomposition for Global Solutions}

We establish sequential soliton resolution when $T_+=\infty$ in this section.

\begin{prop}[Sequential decomposition for a global-in-time solution]\label{prop:global-seq-dec}
 Let $u_0\in \calE$, and let $u(t)$ denote the solution to $\eqref{eqn:GL}$ with initial data $u_0$. Suppose that $T_+(u_0)=\infty$. Then, there exist a function $u^*\in \calE$ , an integer $N\geq 0$, a sequence of time $\{t_n\}_{n\geq 1}$ such that $t_n \to \infty$, a sequence of phases $\vec{\theta_n}=(\theta_{1,n},\dots, \theta_{N,n})\in \bbT^N$, a sequence of scales $\vec{\lam_n}=(\lam_{1,n},\dots, \lam_{N,n})\in (0,\infty)^N$, and an error $g_n$ defined by 
\[
u(t_n)=\sum_{j=1}^N e^{i\theta_{j,n}}W_{\lambda_{j,n}}+g_n
\]
with the following properties:
\begin{enumerate}[label={(\roman*).}]
    \item  The integer $N\geq 1$  satisfies,
    \begin{align} 
     \label{eq:global-energy-limit} 
\lim_{t \to \infty} E( u(t)) = N E(W)  ; 
\end{align}
    \item for any $\al>0 $ 
\begin{align} 
\label{eq:global-en-ext-bu}
\lim_{t\to \infty} E(u(t); \al (T_+ -t)^{\frac{1}{2}}, \infty) = 0,
\end{align}

and there exists $0< T_0 < \infty$ and function $\rho : [T_0, \infty) \to (0,\infty)$ satisfying, 
\begin{equation} \label{eq:global-radiation} 
\lim_{t \to \infty} \big((\rho(t) / \sqrt{t}) + \| u(t) -  u^*\|_{\cE(\rho(t))}\big) = 0; 
\end{equation}

    \item the error $g_n$ and the scales $\vec \lam_n$ satisfy, 
\begin{align} \label{eq:global-d(t_n)} 
\lim_{ n \to \infty} \Big( \| g_n \|_{\E}^2 + \sum_{j =1}^N \Big( \frac{\lam_{n, j}}{\lam_{n, j+1}} \Big)^{\frac{D-2}{2}}  \Big)^{\frac{1}{2}}  = 0 , 
\end{align}
where here we adopt the convention that $\lam_{n, N+1} := t_n^{\frac{1}{2}}$. 
    
\end{enumerate}

\end{prop}
\begin{proof}[Proof of Proposition~\ref{prop:global-seq-dec}]

    Let \( u(t) \in \mathcal{E} \) be a solution of $\eqref{eqn:GL}$ defined globally in time. By $\eqref{eq:tension-L2}$ we can find a sequence \( t_n \to \infty \) so that, \( t_n^{\frac{1}{2}} \|\mathcal{T}(u(t_n))\|_{L^2} \to 0 \) as \( n \to \infty \). We can now apply Lemma $\ref{lem:compactness}$ which yields \( N \geq 0 \), \( \vec{\te} \in \bbT^N \), \( \vec{\lambda}_n \in (0,\infty)^N \) such that after passing to a subsequence, we have
\begin{align}
\label{eqn:con-in-global}
\lim_{n \to \infty} \left( \|u(t_n) - \mathcal{W}(\vec{\te}, \vec{\lambda}_n)\|_{\mathcal{E}(r \leq A\sqrt{t_n})}^2 + \sum_{j=1}^{N-1} \left( \frac{\lambda_{n,j}}{\lambda_{n,j+1}} \right)^{\frac{D-2}{2}} \right) = 0 
\end{align}
for each \( A > 0 \), and moreover that \( \lambda_{n,N} \lesssim t_n^{\frac{1}{2}} \). Fix \( \alpha > 0 \) and let \( \varepsilon > 0 \) be small enough such that by $\eqref{eq:tension-L2}$ and the fact that \( \|u_0\|_{\mathcal{E}} < \infty \) we can find \( T_0 = T_0(\varepsilon) > 0 \) such that,
\begin{equation}
\label{eqn:control-tau-L2}
    \frac{4\|u_0\|_{\mathcal{E}}}{\alpha} \left( \int_{T_0}^{\infty} \int_0^\infty |\partial_t u(t, r)|^2 r^{D-1} \mathrm{d}r\mathrm{d}t \right)^{\frac{1}{2}} \leq \varepsilon. 
\end{equation}

Next, choose \( T \geq T_0 \) so that
\begin{equation}
\label{eq:small-out}
    \|u(T_0)\|_{\mathcal{E}(\alpha\sqrt{T}/4, \infty)} \leq \varepsilon.
\end{equation}

Fixing any such \( T \), we set
\[
\phi(t, r) = \phi_T(r) = 1 - \chi(4r/\alpha\sqrt{T}) \text{ for } t \in [T_0, T]
\]

where \( \chi(r) \) is a smooth function on \( (0, \infty) \) such that \( \chi(r) = 1 \) for \( r \leq 1 \), \( \chi(r) = 0 \) if \( r \geq 4 \), and \( |\chi'(r)| \leq 1 \) for all \( r \in (0, \infty) \). Since \( \frac{\mathrm{d}}{\mathrm{d}t}\phi(t, r) = 0 \) for \( t \in [T_0, T] \) it follows from Lemma $\ref{lem:localized-energy}$ that
\[
\begin{aligned}
&\int_0^\infty \tilde{\mathbf{e}}(u(T))\phi^2 \ud r - \int_0^\infty \tilde{\mathbf{e}}(u(T_0))\phi^2 \ud r\\
&\lesssim -2  \int_{T_0}^T \int_0^\infty |\partial_t u|^2\phi^2 \ud r \ud t+ 2 \left(  \int_{T_0}^T \int_0^\infty |u|^{2p}\phi^2 \ud r \ud t \right)^{1/2} \left(  \int_{T_0}^T \int_0^\infty |\partial_t u|^2\phi^2 \ud r\ud t \right)^{1/2} \\
&\quad + 4 \left(  \int_{T_0}^T \int_0^\infty  |\partial_t u|^2\phi^2(\partial_r \phi)^2 \ud r \ud t\right)^{1/2} \left(  \int_{T_0}^T \int_0^\infty |\partial_r u|^2 \ud r \ud t \right)^{1/2} 
\\
&\quad
+ 2 \left( \int_{T_0}^T \int_0^\infty \frac{|u|^2}{r^4}\phi^2 \ud r \ud t \right)^{1/2} \left( \int_{T_0}^T \int_0^\infty |\partial_t u|^2\phi^2 \ud r \ud t \right)^{1/2} \\
&\lesssim \left( \int_{T_0}^T \int_{\frac{\alpha\sqrt{T}}{4}}^\infty \frac{1}{r^3}\mathrm{d}r\mathrm{d}t \right)^{1/2} \left( \int_{T_0}^T \int_0^\infty |\partial_t u|^2 \ud r \ud t\right)^{1/2} + \frac{1}{\al}\left( \frac{T - T_0}{T} \right)^{1/2} \left( \int_{T_0}^T \int_0^\infty  |\partial_t u|^2 \ud r \ud t\right)^{1/2} \\
&\lesssim \frac{1}{\al}\left( \frac{T - T_0}{T} \right)^{1/2}\left( \int_{T_0}^\infty \int_0^\infty |\partial_t u|^2 \right)^{1/2} 
\end{aligned}
\]
where the constant in the above inequality depends on \( \sup_{t \in [0, \infty)} \|u(t)\|_{\mathcal{E}} \) .  Using the above together with $\eqref{eqn:control-tau-L2}$ and $\eqref{eq:small-out}$, we find that the Dirichlet energy can be made arbitrarily small
\[
\ti{E}(u(T); \alpha\sqrt{T}, \infty) \lesssim \varepsilon
\]
for all sufficiently large $T$ . Thus by the localized coercivity lemma $\ref{lem:localized-coercivity}$ we see that \( E(u(T); \alpha\sqrt{T}, \infty) \geq 0 \) and moreover since
\[
E(u(T); \alpha\sqrt{T}, \infty) \lesssim \ti{E}(u(T); \alpha\sqrt{T}, \infty) \leq \varepsilon
\]
we get $\eqref{eq:global-en-ext-bu}$. Since we proved $\eqref{eq:global-en-ext-bu}$ for any \( \alpha > 0 \) there exists a curve \( \rho \) such that $\eqref{eq:global-radiation}$ holds. Returning to the sequential decomposition we see from $\eqref{eqn:con-in-global}$, the fact that \( \lambda_{n,N} \lesssim t_n^{\frac{1}{2}} \), and from $\eqref{eq:global-en-ext-bu}$ that we must have
\[
\lim_{n \to \infty} \frac{\lambda_{n,N}}{t_n^{\frac{1}{2}}} = 0.
\]
Then, $\eqref{eq:global-d(t_n)}$ follows from the above, $\eqref{eq:global-radiation}$  and $\eqref{eqn:con-in-global}$. Moreover, we see that \( \lim_{n \to \infty} E(u(t_n)) = N E(W) \) and the continuous limit $\eqref{eq:global-energy-limit}$ then follows from the fact that \( E(u(t)) \) is non-increasing.

\end{proof}

\section{Decomposition of the Solution and Collision Intervals}
\label{sec:4}
For the remainder of the paper we fix a solution $u(t)\in \calE$ of $\eqref{eqn:GL}$, defined on the time interval $I_+=[0,T_+)$, where $T_+<\infty$ in the finite time blow-up case and $T_+=\infty$ in the global case. Let $u^*$ be the body map defined in Proposition $\ref{body-map}$ and in case of a global solution we adopt the convention that $u^*=0$. By  Proposition $\ref{prop:finite-seq-dec}$ and $\ref{prop:global-seq-dec}$ there exists an integer $N\geq 0$ and a sequences of times $t_n\to T_+$ such that $u(t_n)-u^*$ approaches an $N$- bubble as $n\to \infty$.

We recall the definition of the localized distance function.

\begin{defn}(Proximity to a multi-bubble)
    For all \( t \in I_+ \), \( \rho \in (0, \infty) \), and \( K \in \{0, 1, \ldots, N\} \), we define the localized multi-bubble proximity function as
\[
\mathbf{d}_K(t; \rho) := \inf_{\vec{\theta}, \vec{\lambda}} \left( \|u(t) - u^* - \mathcal{W}(\vec{\theta}, \vec{\lambda})\|_{\mathcal{E}(\rho)}^2 + \sum_{j=K}^N \left( \frac{\lambda_j}{\lambda_{j+1}} \right)^{\frac{D-2}{2}} \right)^{\frac{1}{2}},
\]
where $u^*=0$ when $T_+=\infty$, \( \vec{\te} := (\te_{K+1}, \ldots, \te_N) \in \bbT^{N-K} \), \( \vec{\lambda} := (\lambda_{K+1}, \ldots, \lambda_N) \in (0, \infty)^{N-K} \), \( \lambda_K := \rho \) and \( \lambda_{N+1} := \sqrt{T_+ - t} \) when \( T_+ < \infty \) and \( \lambda_{N+1} := \sqrt{t} \) when \( T_+ = \infty \). 

The multi-bubble proximity function is defined by \( \mathbf{d}(t) := \mathbf{d}_0(t; 0) \).
\end{defn}
 
The consequence of Proposition $\ref{prop:finite-seq-dec}$ and $\ref{prop:global-seq-dec}$  can be rephrased as follows in this notation: there exists a monotone sequence $t_n\to T_+$ such that 

\begin{equation}
    \label{resolution-sequence}
    \lim_{n\to \infty} \bfd(t_n)=0.
\end{equation}

\subsection{Collision Intervals}
\label{sec:4.1}

To prove Theorem $\ref{thm:main}$ we only need to show that

\begin{equation}
\label{eq:resolution-continuous}
\lim_{t \to T_+} \bfd(t) = 0.
\end{equation}

The approach which we adopt in order to prove~\eqref{eq:resolution-continuous} is to study colliding bubbles.

A collision interval is defined as follows.

\begin{defn}[Collision interval]
\label{def:collision}
Let $K \in \{0, 1, \ldots, N\}$. A compact time interval $[a, b] \subset I_+$ is a \emph{collision interval}
with parameters $0 < \varepsilon < \eta$ and $N - K$ exterior bubbles if
\begin{itemize}
\item $\bfd(a) \leq \varepsilon$ and $\bfd(b) \ge \eta$,
\item there exists a function $\rho_K: [a, b] \to (0, \infty)$ such that $\bfd_K(t; \rho_K(t)) \leq \varepsilon$
for all $t \in [a, b]$.
\end{itemize}
In this case, we write $[a, b] \in \calC_K(\varepsilon, \eta)$.

\end{defn}
\begin{defn}[Choice of $K$]
\label{def:K-choice}
We define $K$ as the \emph{smallest} nonnegative integer having the following property.
There exist $\eta > 0$, a decreasing sequence $\varepsilon_n \to 0$,
and sequences $(a_n), (b_n)$ such that $[a_n, b_n] \in \calC_K(\varepsilon_n, \eta)$ for all $n \in \{1, 2, \ldots\}$.
\end{defn}
\begin{lem}[Existence of $K \ge 1$]
\label{lem:K-exist}
If \eqref{eq:resolution-continuous} is false, then $K$ is well defined and $K \in \{1, \ldots, N\}$.

\end{lem}

\begin{rem} 
The fact that $K \ge 1$ means that at least one bubble must lose its shape if~\eqref{eq:resolution-continuous} is false.
\end{rem} 

\begin{proof}
    The proof is similar to the proof of Lemma 5.6 in \cite{JL23a} and we skip it.
\end{proof}

\begin{rem} \label{rem:collision} 
For each collision interval  we may assume without loss of generality that $\bfd(a_n)  = \veps_n$,  $\bfd(b_n) = \eta$, and $\bfd(t) \in [\veps_n, \eta]$ for each $t \in [a_n, b_n]$. Indeed, given some initial choice of $[a_n, b_n] \in \calC_K( \veps_n, \eta)$, 
just set $a_n \le \ti a_n := \sup\{ t \in [a_n,  b_n] \mid \bfd(t) \le \veps_n \}$ and $\ti b_n := \inf\{t \in [\ti a_n, b_n] \mid \bfd(t) \ge \eta\}$. 

Similarly, given some initial choice $\veps_n \to 0, \eta>0$ and intervals $[a_n, b_n] \in \calC_K( \eta, \veps_n)$ we are free to ``enlarge'' $\veps_n$ or ``shrink'' $\eta>0$, by choosing some other sequence $\veps_n \le \ti \veps_n  \to 0$, and $0< \ti \eta \le \eta$, and new collision subintervals $[\ti a_n, \ti b_n]  \subset [a_n, b_n] \cap \calC_{K}(\ti \eta, \ti \veps_n)$ as in the previous paragraph. We will enlarge our initial choice of $\veps_n$ and shrink $\eta$ in this fashion  over the course of the proof. 
\end{rem}

\subsection{Decomposition of the Solution}
\begin{lem}[Basic modulation]\label{lem:basic-mod}
Let $K \ge 1$ be the number given by Lemma~\ref{lem:K-exist}. There exist $\eta>0$, a sequence $\veps_n \to 0$,  and sequences $a_n, b_n \to \infty$ satisfying the requirements of Definition~\ref{def:K-choice}, and such that $\bfd(a_n)  = \veps_n$, $\bfd(b_n) = \eta$ and $\bfd(t) \in [\veps_n, \eta]$ for all $t \in [a_n, b_n]$ and so that the following properties hold. There exist phases $\vec \theta \in C^1(\cup_{n \in \N} [a_n, b_n];  \bbT^{N})$,  a function $\vec \lambda = ( \lambda_1, \dots, \lam_N) \in C^1(\cup_{n \in \N} [a_n, b_n];  (0, \infty)^{N})$, sequences $\al_n \to 0$ and $\nu_n \to 0$,  such that defining the functions, 
\begin{align} \label{eq:nu-def} 
\nu:\cup_{n \in \N} [a_n, b_n] \to (0, \infty), \quad  \nu(t):= \nu_n \lam_{K+1}(t), \mfor \, \, t\in[a_n,b_n],
\end{align}
\begin{align}
\al : \cup_{n \in \N} [a_n, b_n] \to (0, \infty), \quad  \al(t):= \begin{cases}  \al_n\sqrt{T_+ - t_n}\mif T_+<\infty \\ \al_n \sqrt{t} \mif T_+ = \infty \end{cases}, \mfor \, \, t\in[a_n,b_n], 
\end{align}
\begin{align} 
\label{eq:u^*(t)-def} 
u^*(t) :=  \begin{cases} (1 - \chi_{\al(t)})  u(t) \mif T_+<\infty \\ 0 \mif T_+= \infty \end{cases}  
\end{align}
and 
\begin{align}
\label{eq:g-def-1}
g: \cup_{n \in \N} [a_n, b_n] \to  \E; \quad g(t) := u(t) - u^*(t) - \calW(\vec \theta(t), \vec \lambda(t)), 
\end{align}
there hold,
\begin{itemize} 
\item[(i)] the orthogonality conditions, 
\begin{align}
\label{eq:g-ortho} 
0  = \La i e^{i\theta_j} \calZ_{1,\lambda_j}   \mid g\Ra=\La e^{i\theta_j} \calZ_{2,\lambda_j}    \mid g\Ra  \  , \quad \forall j = 1, \dots, M, , t \in [a_n, b_n], \quad \forall n; 
\end{align}
\item[(ii)] and the estimates, 
\begin{align}
\label{eq:nu-estimates} 
\lim_{n \to \infty} \sup_{t \in [a_n, b_n]} \Big(  \frac{\nu(t)}{ \lam_{K+1}(t)}  + \sum_{j=K+1}^{N-1}  \frac{\lam_{j}(t)}{\lam_{j+1}(t)}  + \frac{\lam_{N}(t)}{\al(t)}+ E( u(t); \frac{1}{4}{\nu(t)},  4 \nu(t))  \Big) = 0,
\end{align}
\begin{align}
\label{eq:d-g-lam} 
C_0^{-1}\bfd(t)  \leq \| g(t) \|_{\cE} +  \sum_{j=1}^{N-1} \Big( \frac{ \lam_{j}(t)}{\lam_{j+1}(t)} \Big)^{\frac{D-2}{4}} 
\leq C_0\bfd(t), 
\end{align} 

\begin{align}
\label{eq:lambda'-bound} 
\abs{\lam_j'(t)} \le C_0 \frac{1}{\lam_j(t)}  \bfd(t), 
\end{align} 
for all $t \in [a_n, b_n]$ and all $n \in \N$; 
\item[(iii)] for any sequence $s_n \in [a_n, b_n]$ and any sequence $R_n$ such that $\nu(s_n) \le R_n \ll \lam_{K+1}(s_n)$  if $K < N$ and $\nu(s_n) \le R_n \le \al(s_n)$  if $K=N$, then, 
\begin{align}
\label{eq:N-K-bubbles} 
\lim_{n \to \infty} E( u(s_n); R_n, \infty) = (N-K) E(Q) + E( u^*). 
\end{align} 
and, 
\begin{align}
\label{eq:N-K-converge} 
\lim_{n \to \infty}  \Big( \| u(s_n) - u^*(s_n) - \calW(\theta_{K+1}, \dots, \theta_N, &\lam_{K+1}(s_n), \dots, \lam_{N}(s_n)) \|_{\E( r \ge R_n)} \\
& + \sum_{j=K+1}^N \Big( \frac{\lam_{j}(s_n)}{\lam_{j+1}(s_n)} \Big)^{\frac{D-2}{4}} \Big) = 0. 
\end{align} 
\end{itemize} 
\end{lem}

\begin{rem}
    We will use $u^*(t)$ and $u^*$ to distinguish the two functions.
\end{rem}

\begin{rem}
    The only effect of the basic modulation is to give a simple control of the dynamics of scales. We do not need any coercive estimates produced by the existence of discrete spectrum. 
\end{rem}

\begin{proof}
The proof of $\eqref{eq:g-ortho}$, $\eqref{eq:nu-estimates}$,$\eqref{eq:d-g-lam}$, $\eqref{eq:N-K-bubbles}$, $\eqref{eq:N-K-converge}$  follows directly from the argument outlined in Step 1 of the proof of Lemma 5.12 in \cite{JL23b}. Thus, it suffices to prove the dynamical estimate $\eqref{eq:lambda'-bound}$.

Differentiating in time the orthogonality conditions $\ref{eq:g-ortho}$ yields, for each $j=1,\dots,N$, the identity
\begin{align} 
\label{eq:g-ortho-1}
& 0=\La i e^{i\theta_j} \calZ_{1,\lambda_j}    \mid \p_tg \Ra- \dot{\theta_j}\La e^{i\theta_j} \calZ_{1,\lambda_j} \mid g\Ra-\frac{\dot{\lambda_j}}{\lambda_j} \La i e^{i\theta_j} \Lam\calZ_{1,\lambda_j}    \mid g \Ra
\\ &=\La e^{i\theta_j} \calZ_{2,\lambda_j}  \mid \p_t g \Ra+  \dot{\theta_j}\La ie^{i\theta_j} \calZ_{2,\lambda_j}  \mid g\Ra -\frac{\dot{\lambda_j}}{\lambda_j} \La e^{i\theta_j} \Lam\calZ_{2,\lambda_j}  \mid  g \Ra
\end{align}

Next, differentiating in time the expression for $g(t)$ in $\ref{eq:g-def-1}$

\begin{align} 
   & \p_t g= \chi_{\al(t)}\p_tu(t)+\p_t\chi_{\al(t)} \cdot u(t) - \sum_{j=1}^N i\dot{\theta_j}e^{i\theta_j} W_{\lambda_j} + \sum_{j=1}^N \frac{\dot{\lambda_j}}{\lambda_j} e^{i\theta_j} \Lam W_{\lambda_j}
   \\ 
   &=  \chi_{\alpha(t)}(z\Delta u+zf(u))-u(t) \frac{\dot{\alpha}(t)}{\alpha(t)} (r\p_r \chi)\left( \cdot / \alpha(t) \right) - \sum_{j=1}^N i\dot{\theta_j}e^{i\theta_j} W_{\lambda_j} + \sum_{j=1}^N \frac{\dot{\lambda_j}}{\lambda_j} e^{i\theta_j} \Lam W_{\lambda_j}
   \\
   &= z\Delta (\chi_\alpha u) + zf\left(\chi_\alpha u\right) - \sum_{j=1}^N i\dot{\theta_j}e^{i\theta_j} W_{\lambda_j} + \sum_{j=1}^N \frac{\dot{\lambda_j}}{\lambda_j} e^{i\theta_j} \Lam W_{\lambda_j}
   \\
   &\quad - zu\Delta \chi_\alpha - 2z\partial_r u \partial_r \chi_\alpha - u(t) \frac{\dot{\alpha}(t)}{\alpha(t)} \left(r\partial_r \chi\right)(\cdot/\alpha(t)) + z\left(f(u)\chi_\alpha - f\left(\chi_\alpha u\right)\right) 
   \\
   &= z\Delta g + zf'\left(\calW(\vec{\theta}, \vec{\lambda})\right)g + zf\left(\calW(\vec{\theta}, \vec{\lambda})\right) - z\sum_{j=1}^N e^{i\theta_j}f\left(W_{\lambda_j}\right) \
   \\
   &\quad + zf\left(\calW(\vec{\theta}, \vec{\lambda}) + g\right) - zf\left(\calW(\vec{\theta}, \vec{\lambda})\right) - zf'\left(\calW(\vec{\theta}, \vec{\lambda})\right)g 
   \\
   &\quad + z\left(-u\Delta \chi_\alpha - 2\partial_r u \partial_r \chi_\alpha - \bar{z}u(t) \frac{\dot{\alpha}(t)}{\alpha(t)} \left(r\partial_r \chi\right)(\cdot/\alpha (t)) + f(u)\chi_\alpha - f\left(\chi_\alpha u\right)\right) 
   \\
   & \quad - \sum_{j=1}^N i\dot{\theta_j}e^{i\theta_j} W_{\lambda_j} + \sum_{j=1}^N \frac{\dot{\lambda_j}}{\lambda_j} e^{i\theta_j} \Lam W_{\lambda_j}
   \\
   &= \calL_{\calW}g+f_{\bfi} (\vec{\theta},\vec{\lambda})+f_{\bfq} (\vec{\theta},\vec{\lambda},g)+\calR(u,\alpha) - \sum_{j=1}^N i\dot{\theta_j}e^{i\theta_j} W_{\lambda_j} + \sum_{j=1}^N \frac{\dot{\lambda_j}}{\lambda_j} e^{i\theta_j} \Lam W_{\lambda_j},
\end{align}

where

\begin{align}
\calL_{\calW} &:= z\Delta +z f'(\calW(\vec{\theta},\vec{\lambda})),\quad f_{\bfi} (\vec{\theta},\vec{\lambda}) := z f\big( \calW( \vec \theta, \vec \lam)\big) - z\sum_{j =1}^{N} e^{i\theta_j}  f(W_{\lam_{j}}) \\
f_{\bfq}(\vec \theta, \vec \lam, g) &:= zf\big( \mathcal{W}( \vec \theta, \vec \lam) + g \big)  - z f\big(\mathcal{W}( \vec \theta, \vec \lam) \big) -  z f'\big( \calW( \vec \theta, \vec \lam)\big) g\\
\calR(u,\alpha)&:=z\left(-u\Delta \chi_\alpha - 2\partial_r u \partial_r \chi_\alpha - \bar{z}u(t) \frac{\dot{\alpha}(t)}{\alpha(t)} \left(r\partial_r \chi\right)(\cdot/\alpha (t)) + f(u)\chi_\alpha - f\left(\chi_\alpha u\right)\right) .
\end{align} 

The subscript  $\bfi$ above stands for ``interaction'' and $\bfq$ stands for ``quadratic''.
For each $j\in \{1,\dots,N\}$, we use the identity $\eqref{eq:g-ortho-1}$ to obtain the following system

\begin{align}
    &-\dot{\theta}_j \left( \La  \calZ_{1,\lambda_j}    \mid W_{\lam_j} \Ra+ \La e^{i\theta_j} \calZ_{1,\lambda_j} \mid g\Ra\right)-\frac{\dot{\lambda_j}}{\lambda_j} \La i e^{i\theta_j} \Lam\calZ_{1,\lambda_j}    \mid g \Ra 
    \\
    &
    +\sum_{k\neq j}
    \La i e^{i\theta_j} \calZ_{1,\lambda_j}    \mid - i\dot{\theta_k}e^{i\theta_k} W_{\lambda_k} +  \frac{\dot{\lambda_k}}{\lambda_k} e^{i\theta_k} \Lam W_{\lambda_k} \Ra
    \\
    &= -\La i e^{i\theta_j} \calZ_{1,\lambda_j}    \mid  \calL_{\calW}g+f_{\bfi} (\vec{\theta},\vec{\lambda})+f_{\bfq} (\vec{\theta},\vec{\lambda},g)+\calR(u,\alpha)  \Ra,
\end{align}

\begin{align}
    & 
    \frac{\dot{\lambda_j}}{\lambda_j} \left( \La  \calZ_{2,\lambda_j}  \mid  \Lam W_{\lambda_j}  \Ra  - \La e^{i\theta_j} \Lam\calZ_{2,\lambda_j}  \mid  g \Ra \right)+  \dot{\theta_j}\La ie^{i\theta_j} \calZ_{2,\lambda_j}  \mid g\Ra 
    \\
    &
    +\sum_{k\neq j} \La e^{i\theta_j} \calZ_{2,\lambda_j}  \mid - i\dot{\theta_k}e^{i\theta_k} W_{\lambda_k} +  \frac{\dot{\lambda_k}}{\lambda_k} e^{i\theta_k} \Lam W_{\lambda_k} \Ra
    \\
    & =  -\La e^{i\theta_j} \calZ_{2,\lambda_j}  \mid \calL_{\calW}g+f_{\bfi} (\vec{\theta},\vec{\lambda})+f_{\bfq} (\vec{\theta},\vec{\lambda},g)+\calR(u,\alpha)  \Ra.
\end{align}

The above system is diagonally dominant for all sufficiently small \( \eta_0 > 0 \) due to $\eqref{aux}$ and $\eqref{eq:d-g-lam}$. Hence, it is invertible. Next, we estimate each term for $k=1,2$
\[
\begin{aligned}
\left| \langle \mathcal{L}_{\mathcal{W}} g \mid \mathcal{Z}_{k,\lambda_j} \rangle \right| &\lesssim  \|g\|_{\mathcal{E}}, \\
\left| \langle f_{\mathbf{i}}(\vec{\te}, \vec{\lambda}) \mid {\mathcal{Z}_{k,\lambda_j}} \rangle \right| &\lesssim  \left( \frac{\lambda_j}{\lambda_{j+1}} \right)^{\frac{D-2}{2}} + \left( \frac{\lambda_{j-1}}{\lambda_j} \right)^{\frac{D-2}{2}}  \lesssim   \mathbf{d}(t)^2 + o_n(1) , \\
\left| \langle f_{\mathbf{q}}(\vec{\te}, \vec{\lambda}, g) \mid {\mathcal{Z}_{k,\lambda_j}} \rangle \right| &\lesssim   \mathbf{d}(t)^2 + o_n(1) , \quad \left| \langle \phi(u, \nu) \mid {\mathcal{Z}_{k,\lambda_j}} \rangle \right| \lesssim  o_n(1),
\end{aligned}
\]
where the first inequality follows by integration by parts and Cauchy-Schwarz, the second inequality follows from a computation analogous to Lemma 2.21 in \cite{JL23b}, the third inequality follows from a Taylor expansion, and the fourth inequality follows from $\eqref{eq:nu-estimates}$. Thus, collecting all the above estimates, we get
\[
|\lambda_j'| \lesssim \frac{1}{\lambda_j} \left( \mathbf{d}(t) + \zeta_{3,n} \right)
\]
for some sequence \( \zeta_{3,n} \to 0 \) as \( n \to \infty \). Then $\eqref{eq:lambda'-bound}$ follows by enlarging \( \varepsilon_n \).

\end{proof}

Next, we prove a lemma connecting localized bubbling as in Lemma $\ref{lem:compactness}$ to sequential bubbling. It can be seen through the proof that the phase parameters $\vec \te$ should be time-dependent. 

\begin{lem}\label{lem:delta-to-d}
There exists a constant $\eta_0>0$ having the following property. Let $t_n \in\left[a_n, b_n\right]$ and let $\mu_n$ be a positive sequence satisfying the conditions:
\begin{enumerate}
    \item $\lim _{n \rightarrow \infty} \frac{\mu_n}{\mu_{K+1}\left(t_n\right)}=0$,
    \item $\mu_n \geq \nu\left(t_n\right)$ or $\left\|{u}\left(t_n\right)\right\|_{\mathcal{E}\left(\mu_n, \nu\left(t_n\right)\right)} \leq \eta_0$
    \item $\lim _{n \rightarrow \infty} \delta_{\mu_n}\left(t_n\right)=0$, where $\delta_{\mu_n}$ is defined in $\ref{def:delta}$.
\end{enumerate}
Then $\lim _{n \rightarrow \infty} \mathbf{d}\left(t_n\right)=0$.
\end{lem}
\begin{proof}

Let $R_n$ be a sequence such that $\mu_n \ll R_n \ll \mu_{K+1}\left(t_n\right)$. Without loss of generality, we can assume $R_n \geq \nu\left(t_n\right)$ since $\nu(t)\ll \mu_{K+1}(t)$. Using the definition of the localized distance function \eqref{def:delta}  we obtain parameters $M_n\in\N, \vec{\ti\theta}_n\in\bT^{M_n}, \vec{\ti \lambda}_n\in (0,\infty)^{M_n}$ such that
\begin{align}
\label{eqn:loc-decomp}
\left\|u\left(t_n\right)-\mathcal{W}\left(\vec{\ti\theta}_n, \vec{\ti\lambda}_n\right)\right\|_{\E\left(r \leq \mu_n\right)}^2+\sum_{j=1}^{M_n-1}\left(\frac{\ti\lambda_{n, j}}{\ti\lambda_{n, j+1}}\right)^{\frac{D-2}{2}} \rightarrow 0.
\end{align} 
Set
$$
\begin{aligned}
{u}_n^{(i)} & :=\chi_{\frac{1}{2} \mu_n} {u}\left(t_n\right), \quad {u}_n^{(o)}  :=\left(1-\chi_{R_n}\right) {u}\left(t_n\right), \quad {u}_n^{(m)} :={u}\left(t_n\right)-{u}_n^{(i)}-{u}_n^{(o)} .
\end{aligned}
$$
Observe that if $\mu_n$ is a positive sequence such that $\lim _{n \rightarrow \infty} {\delta}_{\mu_n}\left(t_n\right)=0$, then we have
\begin{align}
\label{eqn:mu_n-annulus}
\lim _{n \rightarrow \infty}\|{u}\left(t_n\right)\|_{\mathcal{E}\left(\frac{1}{2} \mu_n, \mu_n\right)}=0.
\end{align} 
Combining this with the decomposition \eqref{eqn:loc-decomp} we have
$$
\lim _{n \rightarrow \infty}\|{u}_n^{(i)}-\mathcal{W}(\vec{\ti\theta}_n, \vec{\ti\lambda}_n)\|_{\mathcal{E}}=0 .
$$
Furthermore observe that if $t_n \in\left[a_n, b_n\right]$ and $\nu\left(t_n\right) \leq R_n \ll \mu_{K+1}\left(t_n\right)$,  then $\bfd_K(t_n;\nu(t_n))\to 0$ implies that
\begin{align}
\label{eqn:Rn-annulus}
\lim _{n \rightarrow \infty}\left\|{u}\left(t_n\right)\right\|_{\mathcal{E}\left(R_n, 2 R_n\right)}=0.
\end{align} 
Thus using the second assumption along with \eqref{eqn:mu_n-annulus} and \eqref{eqn:Rn-annulus} for $n$ large enough we have
\begin{align}
\|{u}_n^{(m)}\|_{\mathcal{E}}\leq 2 \eta_0, 
\end{align} 
which implies from Lemma \ref{lem:basic-trapping} that $0 \leq E({u}_n^{(m)}) \lesssim \eta_0^2$. We also have, again using \eqref{eqn:mu_n-annulus} and \eqref{eqn:Rn-annulus},
\begin{align}
\limsup _{n \rightarrow \infty}|E({u}(t_n))-E({u}_n^{(i)})-E({u}_n^{(m)})-E({u}_n^{(o)})|=0 .
\end{align} 
Combining the above convergence with $\lim _{n \rightarrow \infty} E({u}_n^{(o)})=(N-K) E({W})+E({u}^*)$ we see that $M_n=K$ and $\lim _{n \rightarrow \infty} E({u}_n^{(m)})=0$. Using Sobolev embedding, we get $\lim _{n \rightarrow \infty}\|{u}_n^{(m)}\|_{\mathcal{E}}=0$. Now we choose parameters $\vec { \lam}_n=(\ti \lam_{n,1},...,\ti \lam_{n,K},\lam_{K+1}(t_n),...,\lam_{N}(t_n))$ and $\vec { \te}_n=(\ti \te_{n,1},...,\ti \te_{n,K}, \allowbreak \te_{K+1}(t_n),...,\te_{N}(t_n))$ where $\lam_{j}(t_n)$ and $\te_{j}(t_n)$ are given by Lemma ~\ref{lem:basic-mod} to obtain the conclusion.
\end{proof}

\section{Conclusion of the Proof}
\label{sec:5}
In this section we prove Theorem $\ref{thm:main}$. For the remainder of the paper, when we write $[a_n,b_n]\in \calC_K(\veps_n,\eta)$ we always assume that $\bfd(a_n)=\veps_n$, $\bfd(b_n)=\eta$ and $\bfd(t)\in [\veps_n,\eta]$ for all $t\in [a_n,b_n]$. This assumption is valid by Remark $\ref{rem:collision}$.

Firstly we prove a lemma  relating the scale of the $K$th bubble $\lambda_K$ the lengths of the collision intervals $b_n-a_n$.

\begin{lem} \label{lem:collision-duration} 
If $\eta_0>0$ is small enough, then for any $\eta \in (0, \eta_0]$ there exist $\veps \in (0, \eta)$ and $C_u>0$ with the following property. If $[c, d] \subset [a_n, b_n]$, $\bfd(c)  \le \veps$ and $\bfd(d) \ge \eta$, then, 
\begin{align}
  (d- c)^{\frac{1}{2}} \ge C_u^{-1}  \lam_K(c).
\end{align} 
\end{lem} 

\begin{proof} 

If not, then there exists $\eta>0$, sequences $\veps_n \to 0$, $[c_n, d_n] \subset [a_n, b_n]$, and $C_n \to \infty$ so that $\bfd(c_n)  = \veps_n$, $\bfd(d_n) =\eta$ and 
\begin{align} \label{eq:short-time} 
(d_n - c_n)^{\frac{1}{2}} \le C_n^{-1} \lam_K(c_n).
\end{align} 
We will show that $[c_n, d_n] \in \calC_{K-1}(\zeta_n, \eta)$ for some $\zeta_n\to 0$, hence contradicting the minimality of $K$.

First, using~\eqref{eq:lambda'-bound}
\begin{align} \label{eq:lambda'-dn} 
\abs{\lambda_j(t)^2 - \lambda_j(c_n)^2} \le C_0 (t- c_n)
\end{align} 
for all $t \in [c_n, d_n]$ and all $j=1,\cdots,N.$ Hence, using the contradiction assumption~\eqref{eq:short-time}  we can ensure that for large enough $n$,  
\begin{align}
\label{eq:large-scale}
\frac{3}{4} \le  \frac{ \lam_j(t)}{\lam_j(c_n)}   \le \frac{5}{4} 
\end{align} 
for all $j = K, \dots, N$ and all $t \in [c_n, d_n]$.  Since $\bfd(c_n) \to 0$ we have
\begin{align} \label{eq:lamK0} 
\lim_{n \to \infty} \sup_{t \in [c_n, d_n]} \sum_{j=K}^{N} \Big( \frac{ \lam_j(t)}{\lam_{j+1}(t)} \Big)^{\frac{D-2}{2}}  = 0
\end{align}  
and furthermore there exists a sequence $\{r_n\}$ such that   
\begin{align} \label{eq:rn-choice} 
\lam_{K-1}(c_n) + (d_n - c_n)^{\frac{1}{2}} \ll r_n \ll \lam_K(c_n) \mand \lim_{n \to \infty} \ti E( u(c_n); \frac{1}{8} r_n, 8 r_n)  = 0.
\end{align} 

Now using~\eqref{eq:lambda'-dn} with $j = K-1$ we get  
\begin{align}
\sup_{t \in [c_n, d_n]} \abs{ \lam_{K-1}(t)^2 - \lam_{K-1}(c_n)^2} \lesssim d_n - c_n, 
\end{align} 
which in turn implies
\begin{align}
\sup_{t \in [c_n, d_n]} \frac{\lam_{K-1}(t)}{ r_n} \lesssim \frac{\lam_{K-1}(c_n)}{ r_n} + \frac{(d_n - c_n)^{\frac{1}{2}}}{r_n} \to 0 \mas n \to \infty
\end{align} 
since $r_n$ satisfies~\eqref{eq:rn-choice}. 

Setting $\rho_{K-1}(t) := r_n$ for $t \in [c_n, d_n]$, we have
\begin{align}
    &u(t)-u^*-e^{i\te_K(c_n)}W_{\lam_K(c_n)}-\sum_{j=K+1}^{N}e^{i\te_j(t)}W_{\lam_j(t)}\\&=u^*(t)-u^*+\chi_{r_n}(u(t)-u^*(t))+v(t)-e^{i\te_K(c_n)}W_{\lam_K(c_n)}-\sum_{j=K+1}^{N}e^{i\te_j(t)}W_{\lam_j(t)}, 
\end{align}
where $v(t):= (1- \chi_{r_n}) (u(t)-u^*(t))$. Thus to prove $[c_n, d_n] \in \calC_{K-1}(\zeta_n, \eta)$, we only need to prove that
\begin{align}
\label{eq:K-1-bubble-1}
    \lim_{n\to \infty} \sup_{t\in [c_n,d_n]}\bigg(\|v(t)-e^{i\te_K(c_n)}W_{\lam_K(c_n)}-\sum_{j=K+1}^{N}e^{i\te_j(t)}W_{\lam_j(t)} \|_{\E(r_n,\infty)}^2 \\+\left(\frac{r_n}{\lam_K(c_n)}\right)^{\frac{D-2}{2}}+\left(\frac{\lam_K(c_n)}{\lam_{K+1}(t)}\right)^{\frac{D-2}{2}}+\sum_{j=K+1}^N\left(\frac{\lam_j(t)}{\lam_{j+1}(t)}\right)^{\frac{D-2}{2}}\bigg)=0 .
\end{align}

Firstly using ~\eqref{eq:N-K-converge} we have
\begin{align}
    \lim_{n\to \infty} \sup_{t\in [c_n,d_n]} \|v(t)-\sum_{j=K+1}^{N}e^{i\te_j(t)}W_{\lam_j(t)} \|_{\E(\nu(t),\infty)}=0 .
\end{align}
Since $\nu(t)=\nu_n\lam_{K+1}(t)$ for $t\in [c_n,d_n]$, we have 
\begin{align}
    \left|\frac{\nu^2(t)}{\nu^2(c_n)}-1\right|=\left|\frac{\lam_{K+1}^2(t)}{\lam_{K+1}^2(c_n)}-1\right|\leq C_0\frac{d_n-c_n}{\lam_K(c_n)}\frac{\lam_K(c_n)}{\lam_{K+1}(c_n)}=o_{n}(1),
\end{align}
which implies
\begin{align}
    \lim_{n\to \infty} \sup_{t\in [c_n,d_n]} \|v(t)-\sum_{j=K+1}^{N}e^{i\te_j(t)}W_{\lam_j(t)} \|_{\E(2\nu(c_n),\infty)}=0 .
\end{align}

Next, using $\bfd(c_n)\to 0$ we have
\begin{align}
    \lim_{n\to\infty}\|v(c_n)-e^{i\te_K(c_n)}W_{\lam_K(c_n)}\|_{\E(\frac{1}{2}r_n,4\nu(c_n))}=0.
\end{align}
Combining this with Lemma $\ref{lem:un-seq}$ we obtain
\begin{align}
    \lim_{n\to\infty}\sup_{t\in[c_n,d_n]}\|v(t)-e^{i\te_K(c_n)}W_{\lam_K(c_n)}\|_{\E(r_n,2\nu(c_n))}=0.
\end{align}

Recall that $\bfd(c_n)\to0$, ~\eqref{eq:large-scale} holds for $j=K,...,N$ and ~\eqref{eq:d-g-lam} , we obtain ~\eqref{eq:K-1-bubble-1}, thus  $[c_n, d_n] \in \calC_{K-1}(\zeta_n, \eta)$ for some $\zeta_n\to 0$ for some $\zeta_n\to0$, which leads to a contradiction and completes the proof.
\end{proof} 

\begin{lem}\label{lem:cndn} Let $\eta_0>0$ be as in Lemma~\ref{lem:collision-duration},  $\eta \in (0, \eta_0]$, $\veps_n \to 0$ be some sequence, and let $[a_n, b_n] \in \calC_K(\veps_n, \eta)$. Then, there exist $\veps \in (0, \eta)$,  $n_0 \in \N$,  and $c_n,  d_n \in (a_n, b_n)$ such that for all $n \ge n_0$, we have 
\begin{align}
\label{eq:d>eps} 
\bfd(t) \ge \veps, \quad \forall \, \, t \in [c_n, d_n], 
\end{align}
\begin{align}
\label{eq:dn-cn} 
d_n - c_n = \frac{1}{n} \lam_K(c_n)^2, 
\end{align} 
and 
\begin{align}
 \label{eq:lamKcn} 
\frac{1}{2} \lam_K(c_n) \le \lam_K(t) \le 2\lam_K(c_n) \quad \forall\, \, t \in [c_n, d_n]. 
\end{align} 
\end{lem} 
\begin{proof} 
    We choose $\veps$ given in Lemma ~\ref{lem:collision-duration} and take $c_n=\sup\{t:t\in [a_n,b_n], \bfd(t)\leq \veps\}$ so that $\bfd(c_n)=\veps$. Lemma ~\ref{lem:collision-duration} ensures that for $n$ large enough we have $c_n+\frac{1}{n} \lam_K(c_n)^2<b_n$ and $\left|\frac{\lam_K^2(t)}{\lam_K^2(c_n)}-1\right|\leq \frac{C_0}{n}$ for all $t\in [c_n,c_n+\frac{1}{n} \lam_K(c_n)^2]$, which proves ~\eqref{eq:lamKcn}.

\end{proof} 

\begin{proof}[Proof of Theorem~\ref{thm:main}] We argue by contradiction. Suppose that Theorem~\ref{thm:main} is false, we consider a sequence of disjoint collision intervals $[a_n, b_n] \in \calC_K(\veps_n, \eta)$  as in Lemma ~\ref{lem:basic-mod} and let $\eta>0$ be sufficiently small such that Lemma~\ref{lem:collision-duration} and Lemma~\ref{lem:delta-to-d} hold. Let $\veps>0$, $n_0$, and $[c_n, d_n]$ be as in Lemma~\ref{lem:cndn}. 

We claim that there exists a constant $c_0>0$ such that for every $n \ge n_0$, 
\begin{align}
\label{eq:tension-lower} 
\inf_{t \in [c_n, d_n]} \lam_K(t)^2 \| \p_t u(t) \|_{L^2}^2 \ge c_0. 
\end{align} 
If not, we can find a sequence $s_n \in [c_n, d_n]$  such that
\begin{align}
\lim_{n \to \infty} \lam_K(s_n) \| \p_t u(s_n) \|_{L^2}  = 0.
\end{align} 
up to a subsequence.
However, then  Lemma~\ref{lem:compactness} yields a sequence $r_n \to \infty$ such that, up to a subsequence,  
\begin{align}
\lim_{n \to \infty}  \delta_{r_n \lam_K(s_n)} ( u(s_n)) = 0.
\end{align} 
Then Lemma~\ref{lem:delta-to-d} implies that 
\begin{align}
\lim_{n \to \infty} \bfd(s_n) = 0,
\end{align} 
contradicting~\eqref{eq:d>eps}. As a consequence~\eqref{eq:tension-lower} holds. Therefore, using~\eqref{eq:tension-lower},~\eqref{eq:lamKcn}, and~\eqref{eq:dn-cn} we have 
\begin{align}
\sum_{n \ge n_0} \int_{c_n}^{d_n} \| \p_t u(t) \|_{L^2}^2 \, \ud t  \ge \frac{c_0}{4} \sum_{n \ge n_0} \int_{c_n}^{d_n} \lam_K(c_n)^{-2} \,  \ud t \ge \frac{c_0}{4} \sum_{n \ge n_0} n^{-1} = \infty.
\end{align} 
On the other hand,  by ~\eqref{eq:tension-L2} and the fact that the $[c_n, d_n]$ are disjoint, we have, 
\begin{align}
\sum_{n \ge n_0} \int_{c_n}^{d_n} \| \p_t u(t) \|_{L^2}^2 \, \ud t  \le \int_0^{T_+} \| \p_t u(t) \|_{L^2}^2 \, \ud t \lesssim \sup_{t\in[0,T_+)} \|u(t)\|_{\calE}< \infty, 
\end{align} 
which leads to a contradiction. 
\end{proof}

\bibliographystyle{plain}
\bibliography{GL}

\end{document}